\documentclass{article}
\usepackage{amsfonts, amssymb, amsthm, amsmath, mathrsfs, mathtools}
\usepackage[linktocpage]{hyperref}
\usepackage[margin=1in,footskip=0.25in]{geometry}
\usepackage{mathtools}
\usepackage{xcolor}
\usepackage{enumitem}
\usepackage{tabularx}
\usepackage{titling}
\newtheorem{theorem}{Theorem}[section]

\newtheorem{lemma}[theorem]{Lemma}

\let\SS\relax
\makeatletter
\def\foo#1{%
\expandafter\newcommand\csname #1#1\endcsname{\mathbb{#1}}
\expandafter\newcommand\csname cal#1\endcsname{\mathcal{#1}}
\expandafter\newcommand\csname fk#1\endcsname{\mathfrak{#1}}
}
\def\bar#1{%
\expandafter\newcommand\csname fk#1\endcsname{\mathfrak{#1}}
}
\count@=0
\loop
\advance\count@ 1
\edef\x{\@alph\count@}%
\edef\y{\@Alph\count@}%
\expandafter\bar\x
\expandafter\foo\y
\ifnum\count@<26
\repeat
\makeatother
\makeatletter
\ExplSyntaxOn
\newcommand{\DeclareMathOperators}[1]
  {
    \clist_map_inline:nn {#1}
      {
        \exp_args:Nc \DeclareMathOperator {##1} {##1}
      }
  }
\ExplSyntaxOff
\makeatother
\DeclareMathOperators{End, GL, Mor, Aut, Hom, Cost, Clip, FullCost, RV}
\newcommand\f{\varphi}

\newcommand\e{\varepsilon}

\newcommand\inv{^{-1}}
\newcommand{\dd}{\mathrm d}

\newcommand{\br}[1]{\left(#1\right)}
\newcommand{\sbr}[1]{\left[#1\right]}

\newcommand{\abs}[1]{\left|#1\right|}
\newcommand{\norm}[1]{\left\Vert#1\right\Vert}
\newcommand{\pdx}[2]{\frac{\partial#1}{\partial#2}}

\usepackage{xspace}

\DeclareRobustCommand{\th}{%
  \ifmmode
    ^{\mathrm{th}}%
  \else
    \textsuperscript{th}%
  \fi
  \xspace
}

\title{Finite Expression Approximation of High-Dimensional PDEs
Without the Curse of Dimensionality}
\date{September 25, 2026}

\preauthor{\begin{center}}
\postauthor{\end{center}}

\author{%
  Zhi Heng Liu\textsuperscript{1},
  Haizhao Yang\textsuperscript{1,2}\thanks{%
    The work of Haizhao Yang was supported in part by
    ONR Award No.~N00014-23-1-2007,
    DOE (ASCR) Award No.~DE-SC0026052,
    and DARPA Award No.~D24AP00325-00.%
  }\\[0.4em]
  {\small
    \textsuperscript{1}Department of Mathematics;
    \textsuperscript{2}Department of Computer Science}\\
  {\small University of Maryland, College Park}\\
  {\small College Park, MD 20742, USA}\\
  {\small \texttt{ZHL@UMD.EDU}, \texttt{HZYANG@UMD.EDU}}
}

\begin{document}
\maketitle
\begin{abstract}
We establish that finite expressions form a symbolic representation class capable of overcoming the curse of dimensionality for several classes of high-dimensional partial differential equations. For semilinear heat equations, we show how finite expression approximations of the terminal condition and nonlinearity can be propagated through the multilevel Picard framework to produce randomized pointwise approximations with prescribed root-mean-square accuracy. For semilinear Kolmogorov equations with Laplacian diffusion and zero drift, diagonal Black--Scholes equations, and the Laplace Dirichlet problem on a half-space, we construct deterministic finite expression approximations with arbitrarily small spatial $L^p$ error. Under suitable growth and regularity assumptions, the evaluation cost of the resulting approximants is bounded polynomially in the dimension and the reciprocal accuracy. A key ingredient in the nonlinear setting is the construction of finite expressions that approximate the PDE solution while preserving the growth and Lipschitz structures required by the stochastic solution theory. More broadly, our results show that dimension-robust approximation of high-dimensional PDEs is not restricted to conventional neural-network architectures: structured symbolic expressions generated from a fixed dictionary can achieve comparable polynomial-complexity guarantees. This provides a rigorous foundation for finite expression methods as a representation paradigm for high-dimensional scientific computing.
\noindent\textbf{Keywords.} finite expressions, curse of dimensionality, partial differential equations, multilevel Picard methods, Monte Carlo approximation.
\end{abstract}

\section{Introduction}

High-dimensional partial differential equations (PDEs) arise in many areas of scientific computing, including stochastic control, uncertainty quantification, kinetic models, and mathematical finance. Their numerical approximation is notoriously difficult because conventional discretization-based methods typically require computational resources that grow exponentially with the dimension. This phenomenon, commonly referred to as the \emph{curse of dimensionality} \cite{bellman1957dynamic}, has motivated extensive efforts to develop representations and algorithms whose complexity grows only polynomially in the dimension $d$ and the reciprocal accuracy $\varepsilon^{-1}$; see \cite{e2022algorithms} for a survey of recent developments.

A major advance in this direction is the realization that, for several important classes of high-dimensional PDEs, the solution itself admits approximations of polynomial complexity. Multilevel Picard (MLP) methods provide one prominent example. They were developed and analyzed for semilinear parabolic equations in \cite{E_2021,Hutzenthaler_2020} and investigated numerically in \cite{E_2019}, with subsequent extensions to equations with gradient-dependent nonlinearities \cite{Hutzenthaler_2021}, Allen--Cahn equations \cite{Beck_2020}, derivative pricing with default risk \cite{Hutzenthaler_2020_default_risk}, and backward stochastic differential equations \cite{hutzenthaler2021overcomingcursedimensionalitynumerical}. In suitable settings, MLP approximants can further be represented by neural networks, which has led to rigorous results showing that neural-network approximations overcome the curse of dimensionality for semilinear heat equations \cite{Hutzenthaler_2020_rectified}, Kolmogorov equations \cite{Jentzen_2021}, Black--Scholes equations \cite{Grohs_2023,elbrachter2022dnn}, and Poisson Dirichlet problems \cite{grohs2022deepelliptic}. Generalization error bounds for empirical risk minimization in the Black--Scholes setting were studied in \cite{Berner_2020}.

These results suggest a more fundamental question about the representation of high-dimensional PDE solutions. Is dimension-robust approximation inherently tied to neural-network architectures, or can substantially more structured symbolic representations achieve the same qualitative complexity behavior? Neural networks provide highly expressive compositional representations, but they are only one possible way to encode a high-dimensional function. In scientific computing, it is natural to seek representations that are not only expressive but also algebraically structured and directly interpretable as mathematical expressions. This leads to the central question of the present work: whether solutions of high-dimensional PDEs can be approximated by finite symbolic expressions with complexity growing only polynomially in $d$ and $\varepsilon^{-1}$.

The \emph{finite expression method} (FEX) provides a natural framework for addressing this question. A finite expression is constructed from the input variables and constants through finitely many operations selected from a fixed dictionary. Recent work \cite{liang2025finiteexpressionmethodsolving, Song_Cameron_Yang_2025,Hardwick_Liang_Yang_2025,Hardwick_Yang_2026, Jiang_Wang_Yang_Forthcoming, Huynh_Bao_Yang_Zytoon_2026} established strong approximation properties of such expressions for high-dimensional functions on compact domains with various applications. However, the approximation of PDE solutions raises additional difficulties that do not arise in direct function approximation. The solution is determined implicitly through the PDE and, in the problems considered here, through stochastic representations involving trajectories on unbounded state spaces. In nonlinear equations, the solution also enters the nonlinearity itself, so approximation errors must be controlled through a nonlinear stability mechanism. Establishing polynomial-complexity approximation at the level of the PDE solution therefore requires considerably more than a direct application of a universal approximation result.

The main result of this paper is that finite expressions can indeed approximate solutions of several representative classes of high-dimensional PDEs without suffering from the curse of dimensionality. We consider both nonlinear and linear equations, and both parabolic and elliptic problems. For semilinear heat equations, semilinear Kolmogorov equations, diagonal Black--Scholes equations, and the Laplace Dirichlet problem on a half-space, we prove that the corresponding solutions admit finite expression approximations with prescribed accuracy and evaluation cost bounded polynomially in the dimension and the reciprocal accuracy. In particular, for every $\varepsilon\in(0,1]$, the relevant approximation can be chosen with cost bounded by
$C d^a \varepsilon^{-b}$, where $C,a,b>0$ are independent of $d$ and $\varepsilon$.

The proof strategy is designed around the solution representation rather than direct approximation of the PDE in physical space. For semilinear problems, we exploit nonlinear stochastic representations and MLP approximations, while carefully constructing finite expression approximations to initial, terminal, or nonlinear data that preserve the growth and Lipschitz structures required for stability. This makes it possible to propagate approximation errors through the nonlinear solution mechanism without losing polynomial dependence on dimension and accuracy. For the linear problems, stochastic representations reduce the solution to expectations of transformed data, which can then be approximated by finite Monte Carlo averages. Suitable sample realizations are fixed to obtain deterministic finite expressions. Since the stochastic dynamics explore unbounded regions, an additional ingredient throughout the analysis is the separation of approximation error on compact sets from the contribution of the tails, together with quantitative estimates showing that both can be controlled without reintroducing exponential dependence on the dimension.

Theorem~\ref{thm:main} treats semilinear heat equations and provides randomized pointwise approximations with prescribed root-mean-square error at points whose norms remain uniformly bounded in dimension. Theorem~\ref{thm:kolmogorov-main} establishes deterministic spatial $L^p$-approximations for semilinear Kolmogorov equations with zero drift, Laplacian diffusion, and a globally Lipschitz scalar reaction. Theorem~\ref{thm:bs-fex} gives analogous deterministic $L^p$-approximation results for diagonal Black--Scholes equations with uniformly bounded drift and volatility parameters. Finally, Theorem~\ref{thm:elliptic-halfspace-fex} treats the Laplace Dirichlet problem on a half-space and constructs approximations on a fixed interior slab away from the boundary. The spatial approximation results hold for every fixed $p\in[1,\infty)$ under the stated local H\"older regularity and growth assumptions.

From a broader perspective, our results show that polynomial-complexity approximation of high dimensional PDE solutions is not confined to conventional neural-network representations. A structured symbolic representation generated from a fixed dictionary of elementary operations can retain enough expressive power to approximate several nonlinear and linear high-dimensional PDE solution families without exponential dependence on the dimension. This suggests that overcoming the curse of dimensionality is a broader phenomenon associated with suitable compositional representations of the solution, rather than a property specific to one architecture. The present work provides a rigorous foundation for finite expressions as such a representation class and opens the possibility of developing symbolic approaches to high-dimensional scientific computing with both approximation guarantees and explicit mathematical structure.

The cost model used throughout counts the evaluation of the finite expressions and, in the semilinear heat setting, the generation of the random variables required by the MLP recursion. The construction of the data expressions themselves and the selection of suitable sample realizations for the deterministic spatial approximants are not included in the evaluation cost.

Section~\ref{sec:cost-conventions} introduces the finite expression dictionaries, the cost conventions, notation, and compact approximation results. Sections~\ref{sec:semilinear-heat} and~\ref{sec:kolmogorov} treat the semilinear heat and Kolmogorov equations. Section~\ref{sec:stochastic-transfer} develops the stochastic transfer and sampling estimates used in the linear settings, which are then applied to the Black--Scholes and half-space Dirichlet problems.

\section{Conventions}\label{sec:cost-conventions}
We make use of the dictionaries $\calD_0\coloneqq\{+,-,\times,/,\max\{0,\cdot\},\sin(\cdot),2^{(\cdot)}\}$ and $\calD_\sigma\coloneqq\calD_0\cup\{\sigma\}$. Here $\sigma$ is the activation from \cite{shen2022deepnetworkapproximationachieving}. It is defined by
\[
\sigma(x)\coloneqq
\begin{cases}
    \sigma_1(x),& x\in[0,\infty),\\
    \sigma_2(x),& x\in(-\infty,0),
\end{cases}
\]
where $\sigma_1(x)$ is first defined to be $|x|$ on $[-1,1]$ and then extended to the real line with period $2$, and $\sigma_2(x)\coloneqq x/(|x|+1)$.
A finite expression over $\calD\in\{\calD_0,\calD_\sigma\}$ is built from input coordinates and fixed real constants by finitely many operations from $\calD$. Denominators must be nonzero on the domain under consideration.

Each scalar operation from $\calD\in\{\calD_0,\calD_\sigma\}$ will, by definition, have unit cost. Reading coordinates and constants has zero cost, and computed values may be reused. We write $\Cost(\Phi)$ for the number of operations in one evaluation of $\Phi$, counting all coordinates when $\Phi$ is vector-valued. In randomized evaluations, generating one scalar Gaussian random variable with prescribed mean and variance, or one scalar random variable uniformly distributed on $[0,1]$, also has unit cost. Unless stated otherwise, we use $\calD_\sigma$ in Section~\ref{sec:semilinear-heat} and $\calD_0$ elsewhere. An avoidance of the curse of dimensionality means that the stated error requirement can be met with $\Cost(\Psi_{d,\e})\leq Cd^a\e^{-b}$ or $\FullCost_{d,\e}\leq Cd^a\e^{-b}$ for deterministic approximants or the prescribed semilinear heat equation estimator, respectively. Here $C,a,b>0$ are independent of $d$ and $\e$ for the dimensions and accuracies covered by the theorem, but may depend on other problem parameters. The exact quantitative dependence on these other parameters is outside the scope of this paper. Our constructive proofs nevertheless provide some insights, which are not emphasized.

The bounds on cost count exact real operations once the data expressions and fixed parameters are supplied. For randomized MLP evaluations in Sections~\ref{sec:semilinear-heat} and~\ref{sec:kolmogorov}, we count random variable generation and all recursive operations, with a deterministic bound for every run. For deterministic spatial approximants, the samples are fixed parameters, and selecting them is not included in the evaluation cost. We remark on this further: the MLP estimator depends on problem data, and the cost analysis in \cite{Hutzenthaler_2020} counts scalar Gaussian and uniform random variable generation, whereas the operation cost of evaluating the terminal condition and nonlinearity at the sampled arguments is not included. We aim to bridge this gap by including the evaluation cost of the approximating data expressions. For deterministic spatial approximants, the samples are fixed parameters. Constructing the data expressions and finding suitable samples are excluded. Our approach does not address finite-precision implementation.

Generic constants $C>0$ may change from line to line with the stated parameter dependence. We now collect some notation used throughout the paper. We adopt the convention $r^0=1$ for $r\geq0$. For $n\in\NN$, let $\mathbf 1_n\coloneqq(1,\ldots,1)\in\RR^n$, and, for any set or event $A$, let $\mathbf 1_A$ denote its indicator. For $\mu\in(0,\infty)$, $\alpha\in(0,1]$, and $D\subseteq\RR^n$, define
\[
    \calH_\mu^\alpha(D)
    \coloneqq
    \left\{
    h\in C(D,\RR):
    \sup_{x\in D}\abs{h(x)}\leq\mu\text{ and }
    \abs{h(x)-h(y)}\leq\mu\norm{x-y}_{\RR^n}^\alpha
    \text{ for all }x,y\in D
    \right\}.
\]
For $a,b\in\RR$ with $a\leq b$, define the projection onto $[a,b]$ by $P_{[a,b]}(y) \coloneqq a+\max\{0,y-a\}-\max\{0,y-b\}$. The map $P_{[a,b]}$ is $1$-Lipschitz, equals the identity on $[a,b]$, and satisfies $\abs{x-P_{[a,b]}(y)}\leq\abs{x-y}$ for $x\in[a,b]$ and $y\in\RR$.
For $K>0$, we abbreviate $\pi_K\coloneqq P_{[-K,K]}$. If $E\colon D\to[0,\infty)$, define the envelope clipping map via $\Clip_E(x,y) \coloneqq P_{[-E(x),E(x)]}(y)$ for $(x,y)\in D\times\RR$. For $n\in\NN$, define the coordinatewise projection $\Pi_n\colon\RR^n\to[0,1]^n$ by $(\Pi_n(x))_i \coloneqq P_{[0,1]}(x_i) =\max\{0,x_i\}-\max\{0,x_i-1\}$ for $i=1,\ldots,n$. The identities $\abs{z}=\max\{0,z\}+\max\{0,-z\}$, $\min\{a,b\}=a-\max\{0,a-b\}$, and $\max\{a,b\}=b+\max\{0,a-b\}$ show that absolute value, minimum, maximum, and the preceding projection maps are finite expressions over $\calD_0$. For either dictionary $\calD$, the map $\Clip_E$ is therefore a finite expression over $\calD$ whenever $E$ is. Clipping to a fixed interval adds only a universal number of operations, whereas envelope clipping additionally depends on the evaluation cost of $E$. Finite expressions over $\calD$ are closed under well-defined composition and finite linear combinations. For a scalar-valued expression $F$ and a vector-valued expression $H$ whose range lies in the domain of $F$, evaluation with reuse gives $\Cost(F\circ H)\leq\Cost(H)+\Cost(F)$. For scalar-valued expressions $F_1,\ldots,F_n$ on a common domain, $n\in\NN$, and fixed real coefficients $a_1,\ldots,a_n$, multiplying and summing the computed values gives $\Cost\br{\sum_{j=1}^n a_jF_j} \leq2n-1+\sum_{j=1}^n\Cost(F_j)$. We use \cite[Theorem~5]{liang2025finiteexpressionmethodsolving} to approximate bounded H\"older data. Let $\SS_k$ be the finite expressions over $\calD_0$ using at most $k$ operators. For $p\in[1,\infty)$, $\alpha\in(0,1]$, $\mu>0$, $f\in\calH_\mu^\alpha([0,1]^d)$, and $\varepsilon\in(0,1]$, the theorem gives $k\in\NN$ and $\phi\in\SS_k$ such that $\|f-\phi\|_{L^p([0,1]^d)}\leq \varepsilon$ with $k= O\left( d^2\br{1+\log d+\log(\varepsilon^{-1})}^2 \right)$. The implicit constant depends only on $\alpha$, $\mu$, and $p$. This bound motivates our choice of $\calD_0$. For continuous data, \cite[Theorem~1]{shen2022deepnetworkapproximationachieving} gives, for $a<b$, $f\in C([a,b]^d)$, and $\e>0$, a neural network $\phi$ with activation $\sigma$ and size $O(d^4)$, hence a finite expression over $\calD_\sigma$, with $\sup_{x\in[a,b]^d}|f(x)-\phi(x)|\leq\e$.

\section{Semilinear heat equations}\label{sec:semilinear-heat}
We approximate the terminal condition and nonlinearity by finite expressions and use the MLP estimator of \cite{Hutzenthaler_2020}. The result bounds the root-mean-square error at $\xi_d$ and includes data evaluation in the cost.
We use the parameters, data, and stochastic objects of Theorem~\ref{thm:main}, except where a lemma specifies different growth or Lipschitz constants. For a scalar random variable $V$ and $r\in[1,\infty)$, we write $\norm{V}_r\coloneqq\br{\EE[|V|^r]}^{1/r}$.
Unless stated otherwise, constants implicit in $O(\cdot)$ may depend on $T$, $L$, $p$, and $B$, but not on $d$, $R_d$, or any approximation tolerance. Set $B\coloneqq \sup_{d\in\NN}\norm{\xi_d}_{\RR^d}$ and $q\coloneqq 2\max\left\{1,\left\lceil \frac p2\right\rceil\right\}$. For $\alpha\in[0,\infty)$, define $A_{\alpha,d}\coloneqq 1+B^\alpha+T^{\alpha/2}d^{\alpha/2}$. For $R_d>0$, let $B_{R_d}\coloneqq\{x\in\RR^d\mid \norm{x}_{\RR^d}\leq R_d\}$ be the closed ball centered at the origin. We also write $X_s\coloneqq \xi_d+W_s^{d,0}$ for $s\in[0,T]$.
For the polynomial envelope used below, set $E_d(x)\coloneqq 2L(1+\norm{x}_{\RR^d}^q)$ for $x\in\RR^d$. Since $q\geq p$, the elementary estimate $1+r^p\leq2(1+r^q)$ for $r\geq0$ implies that the standing growth assumption gives $|g_d(x)|\leq E_d(x)$ and $|f_d(t,x,0)|\leq E_d(x)$.
Recall from Section~\ref{sec:cost-conventions} the clipping map $\Clip_{E_d}$ and, for $K>0$, the projection $\pi_K=P_{[-K,K]}$.
Finally, define the continuous cutoff
\[
	\chi_{R_d}(x)=\max\left\{0,1-\frac{\max\{0,\norm x_{\RR^d}^2-R_d^2\}}{(R_d+1)^2-R_d^2}\right\}.
\]
Then $\chi_{R_d}=1$ on $B_{R_d}$ and $\chi_{R_d}=0$ outside $B_{R_d+1}$.

\begin{theorem}\label{thm:main}
    Let $T\in(0,\infty)$, $L,p\in[0,\infty)$, set $q\coloneqq 2\max\left\{1,\left\lceil \frac p2\right\rceil\right\}$ and $\Theta\coloneqq\bigcup_{m=1}^{\infty}\ZZ^m$, and let $\xi_d\in\RR^d$, $d\in\NN$, satisfy $\sup_{d\in\NN}\norm{\xi_d}_{\RR^d}<\infty$. Suppose $f_d\colon[0,T]\times\RR^d\times\RR\to\RR$ and $g_d\colon\RR^d\to\RR$ are continuous functions for all $d\in\NN$. For all $t\in[0,T]$, $d\in\NN$, $x\in\RR^d$, and $v,w\in\RR$, assume $|f_d(t,x,0)|+|g_d(x)|\leq L(1+\norm{x}_{\RR^d}^p)$ and $|f_d(t,x,v)-f_d(t,x,w)|\leq L|v-w|$. Let $(\Omega,\calF,\PP)$ be a probability space, let $W^{d,\theta}\colon[0,T]\times\Omega\to\RR^d$ for $d\in\NN$, $\theta\in\Theta$ be independent standard Brownian motions with continuous sample paths, and let $r^\theta\colon\Omega\to[0,1]$, $\theta\in\Theta$, be i.i.d. random variables that are uniformly distributed on $[0,1]$. Assume that $(r^\theta)_{\theta\in\Theta}$ and $(W^{d,\theta})_{d\in\NN,\theta\in\Theta}$ are independent, and define $R_t^\theta\coloneqq t+(T-t)r^\theta$ for $t\in[0,T]$ and $\theta\in\Theta$. Then $R_t^\theta$ is uniformly distributed on $[t,T]$, and one evaluation of $R_t^\theta$ uses one scalar uniform random variable and a fixed number of scalar arithmetic operations. For every $d\in\NN$, there exists a unique at most polynomially growing continuous function $u_d\colon[0,T]\times\RR^d\to\RR$ which is a viscosity solution on $(0,T)\times\RR^d$ of $\frac{\partial u_d}{\partial t}(t,x)+\frac12\Delta_xu_d(t,x)+f_d(t,x,u_d(t,x))=0$, with terminal condition $u_d(T,x)=g_d(x)$ for $x\in\RR^d$, and which satisfies, for every $t\in[0,T]$ and $x\in\RR^d$, the Feynman--Kac fixed-point equation
    \begin{equation}\label{eq:semilinear-fk}
        u_d(t,x)
        =\EE\sbr{g_d(x+W_{T-t}^{d,0})}
        +\int_t^T\EE\sbr{f_d\br{s,x+W_{s-t}^{d,0},u_d(s,x+W_{s-t}^{d,0})}}\dd s.
    \end{equation}
    Moreover, for every $\delta>0$ there exist a constant $C\in(0,\infty)$, depending only on $T,L,p,B,\delta$, and a map $n\colon\NN\times(0,1]\to\NN$ such that for all $d\in\NN$ and $\e\in(0,1]$, one can choose $R_d>0$ and continuous finite expressions over $\calD_\sigma$, $\widetilde f_{d,R_d}\colon[0,T]\times\RR^d\times\RR\to\RR$ and $\widetilde g_{d,R_d}\colon\RR^d\to\RR$
    such that $\widetilde f_{d,R_d}$ is Lipschitz with constant $L$ in the third variable and $|\widetilde f_{d,R_d}(t,x,0)|+|\widetilde g_{d,R_d}(x)|\leq 4L(1+\norm{x}_{\RR^d}^q)$. If $\widetilde U_{n,M}^{d,\theta}\colon[0,T]\times\RR^d\times\Omega\to\RR$ is defined exactly as in \cite[Theorem~1.1]{Hutzenthaler_2020}, with $(\widetilde f_{d,R_d},\widetilde g_{d,R_d})$ in place of $(f_d,g_d)$, and $n_{d,\e}\coloneqq n(d,\e)$, then the randomized output $\calA_{d,\e}(\omega)\coloneqq \widetilde U^{d,0}_{n_{d,\e},n_{d,\e}}(0,\xi_d,\omega)$ for $\omega\in\Omega$ satisfies $\br{\EE[|u_d(0,\xi_d)-\calA_{d,\e}|^2]}^{1/2}\leq\e$. The level $n_{d,\e}$ can be prescribed by the a priori criterion \eqref{eq:heat-mlp-level}, using only $d$, $\e$, and the fixed parameters $T,L,p,B$ without knowing $u_d$.
    Let $\operatorname{FullCost}_{d,n,M}$ denote the quantity defined in Subsection~\ref{subsec:cost}, and set $\operatorname{FullCost}_{d,\e}\coloneqq\operatorname{FullCost}_{d,n_{d,\e},n_{d,\e}}$. The full unit-cost model counts scalar Gaussian and uniform draws and scalar operations from $\calD_\sigma$. Once the finite expression data and the level are specified, every run of the prescribed MLP recursion producing $\calA_{d,\e}$ has deterministic evaluation cost satisfying $\operatorname{FullCost}_{d,\e}\leq C d^{4+q(2+\delta/2)}\e^{-(4+\delta)}$.
\end{theorem}
Let $\widetilde u_d$ denote the solution with the finite expression approximations of the data. The triangle inequality gives $\norm{u_d(0,\xi_d)-\calA_{d,\e}}_2\leq \norm{u_d(0,\xi_d)-\widetilde u_d(0,\xi_d)}_2+\norm{\widetilde u_d(0,\xi_d)-\calA_{d,\e}}_2$. Stability estimates control the first term while \cite[Theorem~1.1]{Hutzenthaler_2020} controls the second once the required growth and Lipschitz bounds on $\widetilde f_{d,R_d}$ and $\widetilde g_{d,R_d}$ have been verified.
\subsection{Moment and tail estimates}\label{subsec:moment-tail}
\begin{lemma}\label{lem:fourthmoment}
Suppose $u_d$ is at most polynomially growing and satisfies the Feynman--Kac representation associated with a pair $(f_d,g_d)$. Let $\alpha\in[0,\infty)$, $G,\Lambda\in[0,\infty)$, and assume $|g_d(x)|+|f_d(t,x,0)|\leq G(1+\norm x_{\RR^d}^\alpha)$ and $|f_d(t,x,v)-f_d(t,x,w)|\leq \Lambda |v-w|$. Then there is a constant $C_\alpha>0$, depending only on $\alpha$, such that for all $s\in[0,T]$, $\norm{u_d(s,X_s)}_4\leq C_\alpha G(1+T)A_{\alpha,d}\exp(\Lambda T)$.
\end{lemma}
\begin{proof}
Let $Z$ be a standard Gaussian random vector in $\RR^d$. For each $r\in[0,T]$, $X_r$ has the same distribution as $\xi_d+\sqrt r Z$. Since $\norm{\xi_d}_{\RR^d}\leq B$, we have $\norm{\norm{X_r}_{\RR^d}^\alpha}_4 \leq C_\alpha\br{B^\alpha+r^{\alpha/2}\norm{\norm Z_{\RR^d}^\alpha}_4} \leq C_\alpha\br{B^\alpha+T^{\alpha/2}d^{\alpha/2}}$. Here we used that $\norm{\norm Z_{\RR^d}^\alpha}_4 \allowbreak=\allowbreak\br{\EE\sbr{\norm Z_{\RR^d}^{4\alpha}}}^{1/4} \allowbreak\leq\allowbreak C_\alpha d^{\alpha/2}$. It follows that for every $r\in[0,T]$, $\norm{1+\norm{X_r}_{\RR^d}^\alpha}_4\allowbreak\leq\allowbreak C_\alpha A_{\alpha,d}$. Consequently, $\norm{g_d(X_T)}_4\leq C_\alpha G A_{\alpha,d}$. Moreover, the growth assumption on $f_d(t,x,0)$ and the Lipschitz condition in the third variable give $\norm{f_d(r,X_r,u_d(r,X_r))}_4 \leq \norm{f_d(r,X_r,0)}_4+\Lambda\norm{u_d(r,X_r)}_4 \leq C_\alpha G A_{\alpha,d}+\Lambda\norm{u_d(r,X_r)}_4$. Put $M(s)\coloneqq \norm{u_d(s,X_s)}_4$. The polynomial-growth assumption and Gaussian moment bounds ensure that all random variables below belong to $L^4(\Omega)$. The Feynman--Kac representation and the Markov property give
\[
	u_d(s,X_s)
	=
	\EE\left[
		g_d(X_T)+\int_s^T f_d(r,X_r,u_d(r,X_r))\dd r
		\,\middle|\,X_s
	\right].
\]
Hence, by the contraction property of conditional expectation and Minkowski's integral inequality,
\begin{align*}
	M(s)&\leq \norm{g_d(X_T)}_4+
	\int_s^T\norm{f_d(r,X_r,u_d(r,X_r))}_4\dd r\\
	&\leq C_\alpha G A_{\alpha,d}+\int_s^T\br{C_\alpha G A_{\alpha,d}+\Lambda M(r)}\dd r\\
	&\leq C_\alpha G(1+T)A_{\alpha,d}+\Lambda\int_s^T M(r)\dd r.
\end{align*}
By Gr\"onwall's inequality, $M(s)\leq C_\alpha G(1+T)A_{\alpha,d}\exp(\Lambda(T-s)) \leq C_\alpha G(1+T)A_{\alpha,d}\exp(\Lambda T)$, which is the desired estimate.
\end{proof}

\begin{lemma}\label{lem:gaussiantail}
Let $\alpha\in[0,\infty)$ and assume $R_d>B+\sqrt{Td}$. Then there exists $C_\alpha\in(0,\infty)$, depending only on $\alpha$, such that $\sup_{s\in[0,T]} \norm{(1+\norm{X_s}_{\RR^d}^\alpha)\mathbf 1_{X_s\notin B_{R_d}}}_2 \leq C_\alpha A_{\alpha,d} \exp\br{-\frac{(R_d-B-\sqrt{Td})^2}{8T}}$. Moreover, random variables $(Y_s)_{s\in[0,T]}$ satisfying $\sup_{s\in[0,T]}\norm{Y_s}_4\leq M_d$ obey $\sup_{s\in[0,T]} \norm{Y_s\mathbf 1_{X_s\notin B_{R_d}}}_2 \allowbreak\leq\allowbreak M_d \exp\br{-\frac{(R_d-B-\sqrt{Td})^2}{8T}}$.
\end{lemma}
\begin{proof}
Let $Z$ be a standard Gaussian random vector in $\RR^d$. Since $z\mapsto\norm z_{\RR^d}$ is $1$-Lipschitz and $\EE[\norm Z_{\RR^d}]\leq\sqrt d$, Gaussian concentration gives, for $r\geq0$, $\PP(\norm Z_{\RR^d}\geq\sqrt d+r)\leq e^{-r^2/2}$. For $s\in(0,T]$, the identity in distribution $X_s\overset{\mathrm d}=\xi_d+\sqrt{s}Z$ and the assumption $R_d>B+\sqrt{Td}$ therefore imply
\begin{align}
\PP(X_s\notin B_{R_d})
\leq\PP\left(\norm Z_{\RR^d}\geq\frac{R_d-B}{\sqrt{s}}\right)
\leq\exp\left[-\frac12\left(\frac{R_d-B}{\sqrt{s}}-\sqrt d\right)^2\right]
\leq\exp\br{-\frac{(R_d-B-\sqrt{Td})^2}{2T}}.\label{eq:heat-gaussian-exit-probability}
\end{align}
For $s=0$, the same bound is immediate since $\norm{X_0}_{\RR^d}\leq B<R_d$. Next, by the elementary inequality $(a+b)^2\leq 2a^2+2b^2$, we find $\norm{(1+\norm{X_s}_{\RR^d}^\alpha)\mathbf 1_{X_s\notin B_{R_d}}}_2^2 \leq 2\PP(X_s\notin B_{R_d}) +2\EE\sbr{\norm{X_s}_{\RR^d}^{2\alpha}\mathbf 1_{X_s\notin B_{R_d}}}$. By Cauchy--Schwarz, $\EE\sbr{\norm{X_s}_{\RR^d}^{2\alpha}\mathbf 1_{X_s\notin B_{R_d}}} \leq \EE\sbr{\norm{X_s}_{\RR^d}^{4\alpha}}^{1/2}\PP(X_s\notin B_{R_d})^{1/2}$. As in Lemma~\ref{lem:fourthmoment}, $\EE\sbr{\norm{X_s}_{\RR^d}^{4\alpha}}^{1/2} \leq C_\alpha\br{B^{2\alpha}+T^\alpha d^\alpha} \leq C_\alpha A_{\alpha,d}^2$. Since $\PP(X_s\notin B_{R_d})\leq1$ and $A_{\alpha,d}\geq1$, we obtain $\norm{(1+\norm{X_s}_{\RR^d}^\alpha)\mathbf 1_{X_s\notin B_{R_d}}}_2^2 \leq C_\alpha A_{\alpha,d}^2\PP(X_s\notin B_{R_d})^{1/2}$. Taking square roots and using \eqref{eq:heat-gaussian-exit-probability} gives
\[
\norm{(1+\norm{X_s}_{\RR^d}^\alpha)\mathbf 1_{X_s\notin B_{R_d}}}_2
\leq C_\alpha A_{\alpha,d}\exp\br{-\frac{(R_d-B-\sqrt{Td})^2}{8T}}.
\]
This proves the first estimate. For the second estimate, H\"older's inequality and \eqref{eq:heat-gaussian-exit-probability} give $\norm{Y_s\mathbf 1_{X_s\notin B_{R_d}}}_2 \leq \norm{Y_s}_4\PP(X_s\notin B_{R_d})^{1/4} \leq M_d \exp\br{-\frac{(R_d-B-\sqrt{Td})^2}{8T}}$, as desired.
\end{proof}

\subsection{Approximation of the terminal condition}\label{subsec:g-approx}
\begin{lemma}\label{lem:gapprox}
Let $R_d>0$ and $\e_{g_d,R_d}>0$. Then there exists a continuous finite expression $\widetilde g_{d,R_d}\colon\RR^d\to\RR$
such that $\norm{(g_d-\widetilde g_{d,R_d})(X_T)\mathbf 1_{X_T\in B_{R_d}}}_2 \leq \e_{g_d,R_d}$. Moreover, $|\widetilde g_{d,R_d}(x)|\leq 2L(1+\norm{x}_{\RR^d}^q)$ for $x\in\RR^d$. In the unit-cost model, one has $\operatorname{Cost}(\widetilde g_{d,R_d})=O(d^4)$.
\end{lemma}
\begin{proof}
Fix $R_d>0$ and $\e_{g_d,R_d}>0$. After the coordinatewise affine rescaling of $[-R_d,R_d]^d$ to $[0,1]^d$, \cite[Theorem~1]{shen2022deepnetworkapproximationachieving} provides a $\sigma$-activated neural network $\widehat g_{d,R_d}\colon\RR^d\to\RR$ such that $\abs{g_d(x)-\widehat g_{d,R_d}(x)}\leq\e_{g_d,R_d}$ for $x\in[-R_d,R_d]^d$. Define $\widetilde g_{d,R_d}(x)\coloneqq\chi_{R_d}(x)\Clip_{E_d}(x,\widehat g_{d,R_d}(x))$. This is continuous. By the setup in Section~\ref{sec:semilinear-heat}, $\abs{g_d}\leq E_d$ and $0\leq\chi_{R_d}\leq1$, with $\chi_{R_d}=1$ on $B_{R_d}\subseteq[-R_d,R_d]^d$. The clipping property in Section~\ref{sec:cost-conventions} therefore gives $\abs{g_d(x)-\widetilde g_{d,R_d}(x)} \leq\abs{g_d(x)-\widehat g_{d,R_d}(x)} \leq\e_{g_d,R_d}$ for $x\in B_{R_d}$. Consequently, $\norm{(g_d-\widetilde g_{d,R_d})(X_T)\mathbf 1_{X_T\in B_{R_d}}}_2 \leq\e_{g_d,R_d}\PP(X_T\in B_{R_d})^{1/2} \leq\e_{g_d,R_d}$. Globally, $\abs{\widetilde g_{d,R_d}(x)}\leq E_d(x)=2L(1+\norm{x}_{\RR^d}^q)$. The network has width $N=36d(2d+1)$ and fixed depth $11$, so its evaluation cost is $O(d^4)$. The postprocessing, including evaluation of $E_d$ and $\chi_{R_d}$, adds $O(d)$ operations. Hence $\widetilde g_{d,R_d}$ is a finite expression with $\Cost(\widetilde g_{d,R_d})=O(d^4)$.
\end{proof}

\subsection{Approximation of the nonlinearity}\label{subsec:f-approx}
\begin{lemma}\label{lem:fapprox}
Let $R_d>0$ and $\e_{f_d,R_d}\in(0,1]$. Let $\widetilde g_{d,R_d}$ be as in Lemma~\ref{lem:gapprox}. Then there exists a continuous finite expression $\widetilde f_{d,R_d}\colon[0,T]\times\RR^d\times\RR\to\RR$ such that the perturbed PDE with data $(\widetilde f_{d,R_d},\widetilde g_{d,R_d})$ admits a unique continuous viscosity solution $\widetilde u_d$ of at most polynomial growth satisfying the corresponding Feynman--Kac representation and, for all $s\in[0,T]$, $\norm{(f_d-\widetilde f_{d,R_d})(s,X_s,\widetilde u_d(s,X_s))\mathbf 1_{X_s\in B_{R_d}}}_2 \leq\e_{f_d,R_d}$. Furthermore, $\widetilde f_{d,R_d}$ is Lipschitz with constant $L$ in the third variable, $|\widetilde f_{d,R_d}(t,x,0)|\leq 2L(1+\norm{x}_{\RR^d}^q)$, and hence $|\widetilde f_{d,R_d}(t,x,y)|\leq2L(1+\norm{x}_{\RR^d}^q)+L|y|$. In the unit-cost model, one can choose $\widetilde f_{d,R_d}$ so that $\operatorname{Cost}(\widetilde f_{d,R_d})=O\br{d^4A_{q,d}^2\e_{f_d,R_d}^{-2}}$.
\end{lemma}
\begin{proof}
Fix $R_d>0$ and $\e_{f_d,R_d}\in(0,1]$. If $L=0$, then the assumptions imply $f_d\equiv0$ and $g_d\equiv0$. In this case, take $\widetilde f_{d,R_d}\equiv0$, and all conclusions are immediate. Hence assume $L>0$. Let $M_{4,d}\coloneqq 4C_qL(1+T)A_{q,d}\exp(LT)$, where $C_q$ is the constant in Lemma~\ref{lem:fourthmoment}. Choose
\begin{equation}\label{eq:heat-nonlinearity-parameters}
	K\coloneqq\max\left\{1,\frac{2LM_{4,d}^2}{\e_{f_d,R_d}}\right\},
	\quad
	\eta\coloneqq\frac{\e_{f_d,R_d}}{12},
	\quad\text{and}\quad
	h_0\coloneqq\frac{\e_{f_d,R_d}}{8L}.
\end{equation}
Set $J\coloneqq\lceil K/h_0\rceil$, let $y_j\coloneqq jK/J$ for $-J\leq j\leq J$, and denote the resulting mesh size by $h\coloneqq\frac KJ\leq h_0$. Then $3\eta+2Lh\leq\frac{\e_{f_d,R_d}}2$, and, since $A_{q,d}\geq1$ and $\e_{f_d,R_d}\leq1$,
\begin{equation}\label{eq:heat-mesh-count}
	J\leq1+\frac K{h_0}
	\leq C A_{q,d}^2\e_{f_d,R_d}^{-2}.
\end{equation}
For each $j$, define $r_j(t,x)\coloneqq f_d(t,x,y_j)-f_d(t,x,0)$. Then $r_0=0$, each $r_j$ is continuous, and $|r_j(t,x)-r_k(t,x)|\leq L|y_j-y_k|$.
After the coordinatewise affine rescaling of the rectangular box $[0,T]\times[-R_d,R_d]^d$ to $[0,1]^{d+1}$, absorbed into the first affine layer, \cite[Theorem~1]{shen2022deepnetworkapproximationachieving} provides $\sigma$-activated neural networks $a_j$ and $b_d$, defined on $[0,T]\times\RR^d$, such that
\[
\sup_{(t,x)\in[0,T]\times[-R_d,R_d]^d}|f_d(t,x,0)-b_d(t,x)|\leq\eta\quad\text{and}\quad
\sup_{(t,x)\in[0,T]\times[-R_d,R_d]^d}|r_j(t,x)-a_j(t,x)|\leq\eta
\]
for every $j$. We set $a_0=0$. Each of these networks has input dimension $d+1$, fixed depth $11$, width $36(d+1)(2(d+1)+1)$, and evaluation cost $O(d^4)$ in the unit-cost model. Define $\widetilde b_d(t,x)\coloneqq \Clip_{E_d}(x,b_d(t,x))$. By the estimate in the setup of Section~\ref{sec:semilinear-heat}, $|f_d(t,x,0)|\leq E_d(x)$, and therefore
\[
\sup_{(t,x)\in[0,T]\times[-R_d,R_d]^d}|f_d(t,x,0)-\widetilde b_d(t,x)|\leq\eta.
\]
Define $A(t,x,y)\coloneqq \min_{-J\leq j\leq J}\left\{a_j(t,x)+L|y-y_j|\right\}$ and $\widetilde A(t,x,y)\coloneqq A(t,x,y)-A(t,x,0)$. Since $A$ is a finite minimum of continuous functions, it is continuous. Moreover, $A$, and hence $\widetilde A$, is Lipschitz with constant $L$ in the third variable. Also $\widetilde A(t,x,0)=0$. Let $r(t,x,y)\coloneqq f_d(t,x,y)-f_d(t,x,0)$. For $(t,x,y)\in[0,T]\times B_{R_d}\times[-K,K]$, we claim that $|\widetilde A(t,x,y)-r(t,x,y)|\leq 2\eta+2Lh$. Indeed, for every $j$, $a_j(t,x)+L|y-y_j|\geq r_j(t,x)-\eta+L|y-y_j|\geq r(t,x,y)-\eta$, so $A(t,x,y)\geq r(t,x,y)-\eta$. On the other hand, choose $j$ such that $|y-y_j|\leq h$. Then $A(t,x,y)\leq a_j(t,x)+L|y-y_j|\leq r_j(t,x)+\eta+Lh \leq r(t,x,y)+\eta+2Lh$. Thus $|A(t,x,y)-r(t,x,y)|\leq\eta+2Lh$. At $y=0$, since $a_0=0$ and $y_0=0$, the same argument gives $|A(t,x,0)|\leq\eta$. Hence $|\widetilde A(t,x,y)-r(t,x,y)|\leq 2\eta+2Lh$. Now define
\begin{equation}\label{eq:heat-nonlinearity-construction}
\widetilde f_{d,R_d}(t,x,y)\coloneqq \widetilde b_d(t,x)+\chi_{R_d}(x)\widetilde A(t,x,\pi_K(y)).
\end{equation}
This function is continuous. Since $\pi_K$ is $1$-Lipschitz and $0\leq\chi_{R_d}\leq1$, we have $|\widetilde f_{d,R_d}(t,x,y)-\widetilde f_{d,R_d}(t,x,z)| =\chi_{R_d}(x)|\widetilde A(t,x,\pi_K(y))-\widetilde A(t,x,\pi_K(z))| \leq L|y-z|$. Also, since $\widetilde A(t,x,0)=0$, $|\widetilde f_{d,R_d}(t,x,0)|=|\widetilde b_d(t,x)|\leq E_d(x)=2L(1+\norm{x}_{\RR^d}^q)$. Together with the Lipschitz estimate, this gives $|\widetilde f_{d,R_d}(t,x,y)|\leq2L(1+\norm{x}_{\RR^d}^q)+L|y|$. Since also $|\widetilde g_{d,R_d}(x)|\leq2L(1+\norm{x}_{\RR^d}^q)$, \cite[Theorem~1.1]{Hutzenthaler_2020} applies with growth exponent $q$ and common parameter $\overline L\coloneqq4L$. Hence the perturbed equation admits a unique continuous viscosity solution $\widetilde u_d$ of at most polynomial growth satisfying the corresponding Feynman--Kac representation.
For $(t,x,y)\in[0,T]\times B_{R_d}\times[-K,K]$, one has $\chi_{R_d}(x)=1$ and $\pi_K(y)=y$, so
\begin{align}
|f_d(t,x,y)-\widetilde f_{d,R_d}(t,x,y)|
\leq |f_d(t,x,0)-\widetilde b_d(t,x)|+|r(t,x,y)-\widetilde A(t,x,y)|
\leq 3\eta+2Lh.\label{eq:heat-nonlinearity-local-error}
\end{align}
Let $\delta\coloneqq3\eta+2Lh$ and $Y_s\coloneqq\widetilde u_d(s,X_s)$. Since $\pi_K(y)=\min\{K,\max\{-K,y\}\}$, one has $|\pi_K(Y_s)|\leq K$ for every outcome, although $Y_s$ itself need not lie in $[-K,K]$. Thus \eqref{eq:heat-nonlinearity-local-error} applies at $(s,X_s,\pi_K(Y_s))$ on the event $X_s\in B_{R_d}$. Moreover, $\pi_K(\pi_K(y))=\pi_K(y)$, and \eqref{eq:heat-nonlinearity-construction} gives $\widetilde f_{d,R_d}(t,x,\pi_K(y)) =\widetilde b_d(t,x)+\chi_{R_d}(x)\widetilde A(t,x,\pi_K(\pi_K(y))) =\widetilde f_{d,R_d}(t,x,y)$. Using this identity and \eqref{eq:heat-nonlinearity-local-error}, on the event $X_s\in B_{R_d}$,
\begin{align*}
|f_d(s,X_s,Y_s)-\widetilde f_{d,R_d}(s,X_s,Y_s)|
&\leq |f_d(s,X_s,Y_s)-f_d(s,X_s,\pi_K(Y_s))|\\
&\quad+|f_d(s,X_s,\pi_K(Y_s))-\widetilde f_{d,R_d}(s,X_s,\pi_K(Y_s))|\\
&\leq L|Y_s-\pi_K(Y_s)|+\delta.
\end{align*}
The projection equals $Y_s$ only when $|Y_s|\leq K$ and, in general, $|Y_s-\pi_K(Y_s)|=\max\{0,|Y_s|-K\}\leq |Y_s|\mathbf 1_{|Y_s|\geq K}$. Consequently,
\begin{equation}\label{eq:heat-nonlinearity-error-split}
\norm{(f_d-\widetilde f_{d,R_d})(s,X_s,Y_s)\mathbf 1_{X_s\in B_{R_d}}}_2
\leq \delta+L\norm{Y_s\mathbf 1_{|Y_s|\geq K}}_2.
\end{equation}
By Lemma~\ref{lem:fourthmoment}, applied to the Feynman--Kac representation associated with $(\widetilde f_{d,R_d},\widetilde g_{d,R_d})$ with $\alpha=q$, $G=4L$, and $\Lambda=L$,
\begin{equation}\label{eq:heat-nonlinearity-fourth-moment}
\norm{Y_s}_4\leq4C_qL(1+T)A_{q,d}e^{LT}=M_{4,d}.
\end{equation}
The growth and Lipschitz bounds used here are independent of $R_d$, $K$, and the approximation tolerances. Thus \eqref{eq:heat-nonlinearity-fourth-moment} is uniform in these parameters, and the choice of $K$ in \eqref{eq:heat-nonlinearity-parameters} is not circular. Therefore $\norm{Y_s\mathbf 1_{|Y_s|\geq K}}_2^2 =\EE\sbr{|Y_s|^2\mathbf 1_{|Y_s|\geq K}} \leq \frac1{K^2}\EE\sbr{|Y_s|^4} \leq \frac{M_{4,d}^4}{K^2}$. Thus
\begin{equation}\label{eq:heat-nonlinearity-truncation-tail}
\norm{Y_s\mathbf 1_{|Y_s|\geq K}}_2\leq \frac{M_{4,d}^2}{K}.
\end{equation}
By \eqref{eq:heat-nonlinearity-parameters} and $h\leq h_0$, $\delta\leq\frac{\e_{f_d,R_d}}2$ and $L\frac{M_{4,d}^2}{K}\leq\frac{\e_{f_d,R_d}}2$. Substituting these bounds and \eqref{eq:heat-nonlinearity-truncation-tail} into \eqref{eq:heat-nonlinearity-error-split} gives $\norm{(f_d-\widetilde f_{d,R_d})(s,X_s,\widetilde u_d(s,X_s))\mathbf 1_{X_s\in B_{R_d}}}_2 \leq \e_{f_d,R_d}$. It remains to give a bound on the cost. The networks $b_d$ and $a_j$, $-J\leq j\leq J$, are $\sigma$-activated networks of evaluation cost $O(d^4)$ each. The identities recalled in Section~\ref{sec:cost-conventions} show that the absolute values and the iterated minimum are finite expressions. The finite minimum over $2J+1$ terms and the remaining postprocessing operations cost $O(J+d)$. Therefore $\operatorname{Cost}(\widetilde f_{d,R_d})=O((J+1)d^4)$. By \eqref{eq:heat-mesh-count}, we obtain $\operatorname{Cost}(\widetilde f_{d,R_d})=O\br{d^4A_{q,d}^2\e_{f_d,R_d}^{-2}}$, as needed.
\end{proof}

\subsection{Cost analysis for the finite expression MLP}\label{subsec:cost}
We use the full-cost convention of Section~\ref{sec:cost-conventions}, assigning unit cost also to each scalar uniform draw. Unlike the random variable counts in \cite{Hutzenthaler_2020}, this includes scalar arithmetic and evaluation of the approximating data. The quantity $\RV_{n,M}$ in \cite[Lemma~3.6]{Hutzenthaler_2020} bounds the number of scalar Gaussian and uniform random variables used to compute one realization of the MLP approximation. The operation costs of evaluating $f_d$ and $g_d$, however, are not included in that count. The recurrence in \cite[Lemma~3.6]{Hutzenthaler_2020} gives
\begin{equation}\label{eq:heat-rv-bound}
	\RV_{n,M}\leq d(5M)^n.
\end{equation}
In our setting, evaluations of the finite expression approximations are also counted. The same recurrence argument applies after replacing the dimension parameter by a quantity that also includes the evaluation costs of the approximating data. Let $C_g\coloneqq \operatorname{Cost}(\widetilde g_{d,R_d})$ and $C_f\coloneqq \operatorname{Cost}(\widetilde f_{d,R_d})$. For $n,M\in\NN$, let $\operatorname{FullCost}_{d,n,M}$ denote the cost of one realization of $\widetilde U^{d,\theta}_{n,M}(t,x)$ at an arbitrary $(t,x)\in[0,T]\times\RR^d$ and $\theta\in\Theta$, using the fixed recursive evaluation in the full unit-cost model. Its operation count is independent of $(t,x)$, $\theta$, and the sampled values, so the same bound applies to every recursive call. Set $\operatorname{FullCost}_{d,-1,M} = \operatorname{FullCost}_{d,0,M} =0$. There exists a universal integer $C_0\in\NN$ such that the recursive definition of the MLP approximation gives
\begin{align}
\operatorname{FullCost}_{d,n,M}
&\leq M^n\br{C_0(d+1)+C_g}\notag\\
&\quad+\sum_{l=0}^{n-1}M^{n-l}\br{C_0(d+1)+2C_f+\operatorname{FullCost}_{d,l,M}+\mathbf 1_{\NN}(l)\operatorname{FullCost}_{d,l-1,M}}.
\label{eq:fullcost-recurrence}
\end{align}
Here the factors $C_0(d+1)$ count the scalar random variables and the arithmetic operations needed to form Brownian increments, random times, shifted inputs, scalar multiples, and Monte Carlo sums. The terms involving $C_g$ and $C_f$ count the corresponding evaluations of $\widetilde g_{d,R_d}$ and $\widetilde f_{d,R_d}$, and the last two terms count the recursive calls.

\begin{lemma}\label{lem:fullcost}
There exists a universal constant $C\in(0,\infty)$ such that for all $d,n,M\in\NN$, $\operatorname{FullCost}_{d,n,M}\leq C(d+C_g+C_f)(5M)^n$.
\end{lemma}
\begin{proof}
Set $D\coloneqq C_0(d+1)+C_g+2C_f$. Then \eqref{eq:fullcost-recurrence} implies
\[
\operatorname{FullCost}_{d,n,M}
\leq DM^n+\sum_{l=0}^{n-1}M^{n-l}\br{D+\operatorname{FullCost}_{d,l,M}+\mathbf 1_{\NN}(l)\operatorname{FullCost}_{d,l-1,M}}.
\]
This recurrence is bounded by a recurrence of the same form as the one in \cite[Lemma~3.6]{Hutzenthaler_2020}, with $\RV_{n,M}$ replaced by $\operatorname{FullCost}_{d,n,M}$ and with the dimension parameter $d$ replaced by $D$. The proof of that lemma therefore gives $\operatorname{FullCost}_{d,n,M}\leq D(5M)^n$. Finally, since $D=C_0(d+1)+C_g+2C_f\leq C(d+C_g+C_f)$ for a universal constant $C$, the desired estimate follows.
\end{proof}

\begin{lemma}\label{lem:heat-mlp-level}
Let $T\in(0,\infty)$, $L\in[0,\infty)$, and $\delta\in(0,\infty)$. There exists a constant $C\in(0,\infty)$, depending only on $T,L,\delta$, such that, for every $\overline c\in[0,\infty)$ and $\eta\in(0,1]$, the level
\[
N\coloneqq\min\left\{n\in\NN:\ n\geq2\quad\text{and}\quad
\overline c\frac{e^{n/2}(1+2LT)^n}{n^{n/2}}\leq\eta\right\}
\]
is well defined and satisfies $(5N)^N\leq C\br{1+\overline c^{2+\delta}}\eta^{-(2+\delta)}$.
\end{lemma}
\begin{proof}
The level $N$ is finite since $e^{n/2}(1+2LT)^n/n^{n/2}\to0$. If $N=2$, the claim follows from $(5N)^N=100$ and $\eta\leq1$. If $N\geq3$, minimality gives $\eta<\overline c\frac{e^{(N-1)/2}(1+2LT)^{N-1}}{(N-1)^{(N-1)/2}}$. Using this inequality in the calculation of \cite[equations~(3.53)--(3.57), in the proof of Corollary~3.7]{Hutzenthaler_2020}, with $c$ replaced by $\overline c$, yields $(5N)^N\eta^{2+\delta} \leq 5e\overline c^{2+\delta} \sup_{j\in\NN} \frac{(j+1)(4+8LT)^{j(2+\delta)}}{j^{j\delta/2}}$. The supremum is finite because the logarithm of its $j$th term tends to $-\infty$. This proves the claim.
\end{proof}

\subsection{Proof of Theorem \ref{thm:main}}\label{subsec:proof}
\begin{proof}
The existence, uniqueness, and polynomial growth of $u_d$, together with the fixed-point representation \eqref{eq:semilinear-fk}, follow from \cite[Theorem~1.1 and equation~(1.4)]{Hutzenthaler_2020}. It remains to prove the pointwise root-mean-square approximation and the full-cost estimate. Fix $\delta>0$, $d\in\NN$, and $\e\in(0,1]$. Set
\begin{equation}\label{eq:heat-data-tolerances}
    \e_{g_d,R_d}\coloneqq\frac{\e e^{-LT}}6
    \quad\text{and}\quad
    \e_{f_d,R_d}\coloneqq\frac{\e e^{-LT}}{6(1+T)}.
\end{equation}
For now, let $R_d>B+\sqrt{Td}$ be arbitrary. Choose $\widetilde g_{d,R_d}$ and $\widetilde f_{d,R_d}$ as in Lemmas~\ref{lem:gapprox} and~\ref{lem:fapprox}, using \eqref{eq:heat-data-tolerances}; in particular, $\e_{f_d,R_d}\in(0,1]$. Let $\widetilde u_d$ be the corresponding continuous viscosity solution of at most polynomial growth, satisfying the Feynman--Kac representation. We will choose $R_d$ after obtaining estimates uniform in the approximation parameters.

Define $\mathcal E_d(t)\coloneqq\norm{u_d(t,X_t)-\widetilde u_d(t,X_t)}_2$ for $t\in[0,T]$. Polynomial growth and Gaussian moments ensure that the norms and time integrals below are finite. Subtracting the Feynman--Kac representations conditional on $X_t$, applying the contraction property of conditional expectation and Minkowski's integral inequality, and using the Lipschitz condition for $f_d$ give
\begin{equation}\label{eq:heat-stability-inequality}
\begin{aligned}
    \mathcal E_d(t)
    &\leq\norm{(g_d-\widetilde g_{d,R_d})(X_T)}_2
    +\int_t^T\norm{(f_d-\widetilde f_{d,R_d})(s,X_s,\widetilde u_d(s,X_s))}_2\dd s
    +L\int_t^T\mathcal E_d(s)\dd s.
\end{aligned}
\end{equation}
By Lemma~\ref{lem:gapprox},
\begin{equation}\label{eq:heat-terminal-compact-error}
    \norm{(g_d-\widetilde g_{d,R_d})(X_T)\mathbf 1_{X_T\in B_{R_d}}}_2
    \leq\e_{g_d,R_d},
\end{equation}
and by Lemma~\ref{lem:fapprox}, for every $s\in[0,T]$,
\begin{equation}\label{eq:heat-source-compact-error}
    \norm{(f_d-\widetilde f_{d,R_d})(s,X_s,\widetilde u_d(s,X_s))\mathbf 1_{X_s\in B_{R_d}}}_2
    \leq\e_{f_d,R_d}.
\end{equation}
We now estimate the complementary terms. The growth bounds give $\abs{g_d(x)-\widetilde g_{d,R_d}(x)}\leq4L(1+\norm{x}_{\RR^d}^q)$, since $1+r^p\leq2(1+r^q)$. Lemma~\ref{lem:gaussiantail}, applied with $\alpha=q$, therefore gives
\begin{equation}\label{eq:heat-terminal-tail-error}
    \norm{(g_d-\widetilde g_{d,R_d})(X_T)\mathbf 1_{X_T\notin B_{R_d}}}_2
    \leq C_{\mathrm{terminal}}A_{q,d}
    \exp\br{-\frac{(R_d-B-\sqrt{Td})^2}{8T}},
\end{equation}
where we fix $C_{\mathrm{terminal}}\coloneqq\max\{1,4LC_q\}$, with $C_q$ the constant in Lemma~\ref{lem:gaussiantail}.
For the nonlinear term, the growth bounds at zero and the Lipschitz conditions for both nonlinearities imply, for $t\in[0,T]$, $x\in\RR^d$, and $y\in\RR$,
\begin{equation}\label{eq:heat-source-difference-growth}
    \abs{f_d(t,x,y)-\widetilde f_{d,R_d}(t,x,y)}
    \leq4L(1+\norm{x}_{\RR^d}^q)+2L\abs{y}.
\end{equation}
Moreover, Lemma~\ref{lem:fourthmoment}, applied to the perturbed solution with $\alpha=q$, $G=4L$, and $\Lambda=L$, gives a fixed constant $C_{\mathrm{moment}}\geq1$, depending only on $T,L,p$, such that
\begin{equation}\label{eq:heat-perturbed-fourth-moment}
    \sup_{s\in[0,T]}\norm{\widetilde u_d(s,X_s)}_4
    \leq C_{\mathrm{moment}}A_{q,d}.
\end{equation}
Both $C_{\mathrm{terminal}}$ and $C_{\mathrm{moment}}$ are independent of $d$, $R_d$, and the approximation tolerances. Set $C_{\mathrm{tail}}\coloneqq C_{\mathrm{terminal}}+2LC_{\mathrm{moment}}\geq1$. Apply the two estimates of Lemma~\ref{lem:gaussiantail} to the polynomial term and the solution term in \eqref{eq:heat-source-difference-growth}, respectively. Together with \eqref{eq:heat-perturbed-fourth-moment}, they give
\begin{equation}\label{eq:heat-source-tail-error}
    \norm{(f_d-\widetilde f_{d,R_d})(s,X_s,\widetilde u_d(s,X_s))\mathbf 1_{X_s\notin B_{R_d}}}_2
    \leq C_{\mathrm{tail}}A_{q,d}
    \exp\br{-\frac{(R_d-B-\sqrt{Td})^2}{8T}}
\end{equation}
for every $s\in[0,T]$. Split each data error in \eqref{eq:heat-stability-inequality} into its parts inside and outside of $B_{R_d}$. Bound the terminal error with \eqref{eq:heat-terminal-compact-error} and \eqref{eq:heat-terminal-tail-error} and the nonlinearity error with \eqref{eq:heat-source-compact-error} and \eqref{eq:heat-source-tail-error}. The triangle inequality and $C_{\mathrm{terminal}}+(T-t)C_{\mathrm{tail}}\leq(1+T)C_{\mathrm{tail}}$ then give
\begin{equation}\label{eq:heat-perturbation-bound}
\begin{aligned}
    \mathcal E_d(t)
    &\leq\e_{g_d,R_d}+(T-t)\e_{f_d,R_d}
    +C_{\mathrm{tail}}(1+T)A_{q,d}\exp\br{-\frac{(R_d-B-\sqrt{Td})^2}{8T}}
    +L\int_t^T\mathcal E_d(s)\dd s.
\end{aligned}
\end{equation}
With this fixed $C_{\mathrm{tail}}$, choose
\begin{equation}\label{eq:heat-radius-choice}
    R_d\coloneqq B+\sqrt{Td}
    +\sqrt{8T\log\br{1+\frac{6C_{\mathrm{tail}}(1+T)e^{LT}A_{q,d}}{\e}}}.
\end{equation}
Then $R_d>B+\sqrt{Td}$ and, by \eqref{eq:heat-radius-choice},
\begin{equation}\label{eq:heat-tail-budget}
\begin{aligned}
    C_{\mathrm{tail}}(1+T)A_{q,d}
    \exp\br{-\frac{(R_d-B-\sqrt{Td})^2}{8T}}
    &=\frac{C_{\mathrm{tail}}(1+T)A_{q,d}}
    {1+6C_{\mathrm{tail}}(1+T)e^{LT}A_{q,d}\e\inv}
    \leq\frac{\e e^{-LT}}6.
\end{aligned}
\end{equation}
Substituting \eqref{eq:heat-data-tolerances} and~\eqref{eq:heat-tail-budget} into \eqref{eq:heat-perturbation-bound}, and using $(T-t)/(1+T)\leq1$, gives $\mathcal E_d(t)\leq\frac{\e e^{-LT}}2+L\int_t^T\mathcal E_d(s)\dd s$. Backward Gr\"onwall's inequality yields
\begin{equation}\label{eq:heat-perturbation-half-error}
    \norm{u_d(0,\xi_d)-\widetilde u_d(0,\xi_d)}_2
    =\mathcal E_d(0)
    \leq e^{LT}\frac{\e e^{-LT}}2
    =\frac\e2.
\end{equation}
The radius and the corresponding data approximations are now fixed. The functions $\widetilde f_{d,R_d}$ and $\widetilde g_{d,R_d}$ are continuous, $\widetilde f_{d,R_d}$ is Lipschitz in the third variable with constant $L$, and
\begin{equation}\label{eq:heat-approximating-data-growth}
    \abs{\widetilde f_{d,R_d}(t,x,0)}+\abs{\widetilde g_{d,R_d}(x)}
    \leq4L(1+\norm{x}_{\RR^d}^q).
\end{equation}
Thus they satisfy the hypotheses of \cite[Theorem~1.1]{Hutzenthaler_2020} with exponent $q$ and common growth and Lipschitz parameter $\overline L=4L$. In the error estimate of \cite[Theorem~3.5]{Hutzenthaler_2020}, we use the actual Lipschitz constant $L$. Define the data-size prefactor by
\begin{equation}\label{eq:heat-data-size-definition}
    c_d\coloneqq e^{LT}\sbr{
    \norm{\widetilde g_{d,R_d}(X_T)}_2
    +\br{T\int_0^T\norm{\widetilde f_{d,R_d}(s,X_s,0)}_2^2\dd s}^{1/2}}.
\end{equation}
The growth bounds \eqref{eq:heat-approximating-data-growth} and the Gaussian moment estimate in the proof of Lemma~\ref{lem:fourthmoment} allow us to fix $C_*\geq1$, depending only on $T,L,p,B$, so that the prefactor in \eqref{eq:heat-data-size-definition} satisfies
\begin{equation}\label{eq:heat-data-size-bound}
    c_d\leq\overline c_d\coloneqq C_*A_{q,d}
\end{equation}
uniformly in $d$ and the approximation parameters. By \eqref{eq:heat-data-size-bound} and \cite[Theorem~3.5 and equation~(3.51)]{Hutzenthaler_2020}, for every $n\in\NN$,
\begin{equation}\label{eq:heat-mlp-error-bound}
    \br{\EE\sbr{\abs{\widetilde u_d(0,\xi_d)-\widetilde U_{n,n}^{d,0}(0,\xi_d)}^2}}^{1/2}
    \leq c_d\frac{e^{n/2}(1+2LT)^n}{n^{n/2}}
    \leq\overline c_d\frac{e^{n/2}(1+2LT)^n}{n^{n/2}}.
\end{equation}
In view of \eqref{eq:heat-mlp-error-bound}, define
\begin{equation}\label{eq:heat-mlp-level}
    n_{d,\e}\coloneqq
    \min\left\{n\in\NN:\ n\geq2\quad\text{and}\quad
    \overline c_d\frac{e^{n/2}(1+2LT)^n}{n^{n/2}}\leq\frac\e2\right\}.
\end{equation}
This minimum is finite because $e^{n/2}(1+2LT)^n/n^{n/2}\to0$ as $n\to\infty$. The fixed choice of $C_*$ makes the criterion depend only on $T,L,p,B,d,\e$, not on the unknown solution or its actual approximation error. Equations~\eqref{eq:heat-mlp-error-bound} and~\eqref{eq:heat-mlp-level} give
\begin{equation}\label{eq:heat-mlp-half-error}
    \br{\EE\sbr{\abs{\widetilde u_d(0,\xi_d)-\widetilde U_{n_{d,\e},n_{d,\e}}^{d,0}(0,\xi_d)}^2}}^{1/2}
    \leq\frac\e2.
\end{equation}
For the following complexity estimates, $C>0$ denotes a generic constant depending only on $T,L,p,B,\delta$; neither $C_{\mathrm{tail}}$ nor $C_*$ is changed. Applying Lemma~\ref{lem:heat-mlp-level} with $\overline c=\overline c_d$ and $\eta=\e/2$ gives $(5n_{d,\e})^{n_{d,\e}} \leq C\br{1+\overline c_d^{2+\delta}}\e^{-(2+\delta)}$. Using \eqref{eq:heat-data-size-bound} and $1\leq A_{q,d}\leq Cd^{q/2}$, we obtain
\begin{equation}\label{eq:heat-level-factor}
    (5n_{d,\e})^{n_{d,\e}}
    \leq C A_{q,d}^{2+\delta}\e^{-(2+\delta)}
    \leq C d^{q(1+\delta/2)}\e^{-(2+\delta)}.
\end{equation}
Equations~\eqref{eq:heat-rv-bound} and~\eqref{eq:heat-level-factor} therefore give $\RV_{n_{d,\e},n_{d,\e}} \leq C d A_{q,d}^{2+\delta}\e^{-(2+\delta)} \leq C d^{1+q(1+\delta/2)}\e^{-(2+\delta)}$. Combining \eqref{eq:heat-perturbation-half-error} and~\eqref{eq:heat-mlp-half-error} by the triangle inequality gives $\br{\EE\sbr{\abs{u_d(0,\xi_d)-\widetilde U_{n_{d,\e},n_{d,\e}}^{d,0}(0,\xi_d)}^2}}^{1/2} \leq\e$. It remains to estimate $\FullCost$. By Lemmas~\ref{lem:gapprox} and~\ref{lem:fapprox}, with the tolerances in \eqref{eq:heat-data-tolerances},
\begin{equation}\label{eq:heat-data-costs}
    C_g=\Cost(\widetilde g_{d,R_d})\leq Cd^4
    \quad\text{and}\quad
    C_f=\Cost(\widetilde f_{d,R_d})\leq Cd^4A_{q,d}^2\e^{-2}.
\end{equation}
With $\FullCost_{d,\e}$ as defined in Theorem~\ref{thm:main}, Lemma~\ref{lem:fullcost}, \eqref{eq:heat-level-factor}, and~\eqref{eq:heat-data-costs} give
\[
    \FullCost_{d,\e}
    \leq C\br{d+C_g+C_f}(5n_{d,\e})^{n_{d,\e}}
    \leq C\br{d+d^4+d^4A_{q,d}^2\e^{-2}}A_{q,d}^{2+\delta}\e^{-(2+\delta)}.
\]
Since $\e\in(0,1]$ and $1\leq A_{q,d}\leq Cd^{q/2}$, we obtain $\FullCost_{d,\e} \allowbreak\leq\allowbreak C d^4A_{q,d}^{4+\delta}\e^{-(4+\delta)} \allowbreak\leq\allowbreak C d^{4+q(2+\delta/2)}\e^{-(4+\delta)}$, which completes the proof.
\end{proof}

\section{Kolmogorov equations}\label{sec:kolmogorov}

We obtain deterministic spatial $L^p$-approximations for semilinear Kolmogorov equations with Laplacian diffusion and zero drift. We approximate the initial condition and scalar nonlinearity by finite expressions, control the change in the solution, and express a selected MLP realization as a finite expression. Lemma~\ref{lem:kolmogorov-mlp} combines the pointwise estimate of \cite[Theorem~3.5]{Hutzenthaler_2020} with spatial integration and clipping to give the required $L^p$-bound.

Let $T>0$, $p\in[1,\infty)$, $\rho\in(0,1]$, $L\in[0,\infty)$, and $C_0,\kappa\in[1,\infty)$. Let $m\in[0,\infty)$ and set $\ell\coloneqq 2\max\left\{1,\left\lceil\frac m2\right\rceil\right\}$. Let $f\colon\RR\to\RR$ be Lipschitz continuous with Lipschitz constant $L$. After increasing $L$ if necessary, assume $L\geq1$ and $\abs{f(0)}\leq L$. For every $d\in\NN$, let $g_d\colon\RR^d\to\RR$ be continuous and assume the global growth bound and local H\"older bound
\begin{equation}\label{eq:kolmogorov-data-assumptions}
\begin{aligned}
    \abs{g_d(x)}
    &\leq C_0d^\kappa\br{1+\norm{x}_{\RR^d}^m}
    && \text{for }x\in\RR^d,\\
    \abs{g_d(x)-g_d(y)}
    &\leq C_0d^\kappa R^\kappa\norm{x-y}_{\RR^d}^{\rho}
    && \text{for }R\geq1\text{ and }x,y\in[-R,R]^d.
\end{aligned}
\end{equation}
Let $(\Omega,\mathcal F,\mathbb P)$ be a probability space. For every $d\in\NN$, let $Y_d\colon\Omega\to[0,1]^d$ be uniformly distributed, and let $B^d=(B_s^d)_{s\in[0,T]}$ and $W^d=(W_s^d)_{s\in[0,T]}$ be standard $d$-dimensional Brownian motions such that $Y_d$, $B^d$, and $W^d$ are mutually independent. We assume that $(\Omega,\mathcal F,\mathbb P)$ also carries the independent Brownian motions and uniform random variables required by the MLP construction in Subsection~\ref{subsec:kolmogorov-mlp}, independently of the preceding objects.

For every $d\in\NN$, let $u_d\colon[0,T]\times\RR^d\to\RR$ be the unique continuous viscosity solution of at most polynomial growth of $\pdx{}{t}u_d(t,x)=\Delta_xu_d(t,x)+f(u_d(t,x))$ for $(t,x)\in(0,T)\times\RR^d$, with initial condition $u_d(0,x)=g_d(x)$ for $x\in\RR^d$. It satisfies, for every $(t,x)\in[0,T]\times\RR^d$, the mild representation
\begin{equation}\label{eq:kolmogorov-mild}
    u_d(t,x)
    =\EE\sbr{g_d(x+\sqrt2W_t^d)}
    +\int_0^t\EE\sbr{f\br{u_d(s,x+\sqrt2W_{t-s}^d)}}\dd s.
\end{equation}
Existence, uniqueness, and the agreement of the viscosity and mild solutions in the class of continuous functions of at most polynomial growth follow from \cite[Theorem~1.1]{Beck_2021}, applied after time reversal with zero drift and diffusion matrix $\sqrt2 I_d$.

For every $d\in\NN$, define $Z_s^d\coloneqq Y_d+\sqrt2B_{T-s}^d$, $s\in[0,T]$. In particular, $Z_T^d=Y_d$, so the spatial error at time $T$ on $[0,1]^d$ equals the error under the law of $Z_T^d$.

\begin{theorem}\label{thm:kolmogorov-main}
    Under the assumptions above, there exist constants $C,\alpha\in(0,\infty)$, depending only on $T,p,\rho,L,C_0,\kappa$, and $m$, such that for every $d\in\NN$ and every $\e\in(0,1]$, there exists a finite expression $\Psi_{d,\e}\colon\RR^d\to\RR$ over $\calD_0$ satisfying $\norm{u_d(T,\cdot)-\Psi_{d,\e}}_{L^p([0,1]^d)}\leq\e$ and $\Cost(\Psi_{d,\e})\leq Cd^\alpha\e^{-\alpha}$.
\end{theorem}

\subsection{Finite expression approximation of the data}\label{subsec:kolmogorov-data}

\begin{lemma}\label{lem:kolmogorov-gaussian-law}
    Let $a,b\in[0,\infty)$ and $r\in[1,\infty)$. There exist constants $C,\alpha\in(0,\infty)$, depending only on $a,b,r,T$, such that, for every $d\in\NN$, $s\in[0,T]$, and $R\geq1$, $\norm{1+\norm{Z_s^d}_{\RR^d}^{a}}_{L^r(\Omega)}\leq\allowbreak Cd^\alpha$ and 
    \[
        \norm{\br{1+\norm{Z_s^d}_{\RR^d}^{a}}
        \mathbf 1_{\{Z_s^d\notin[-R,R]^d\}}}_{L^r(\Omega)}
        \leq Cd^\alpha R^{-b}.
    \]
    Moreover, $Z_s^d$ has a density $r_{d,s}$ satisfying $\norm{r_{d,s}}_{L^\infty(\RR^d)}\leq1$.
\end{lemma}

\begin{proof}
    Since $Y_d\in[0,1]^d$ and $B_{T-s}^d$ is Gaussian, standard Gaussian moment estimates give the first bound. If $Z_s^d\notin[-R,R]^d$, then $\norm{Z_s^d}_{\RR^d}>R$. For $b>0$, $\mathbf 1_{\{Z_s^d\notin[-R,R]^d\}} \leq\left(\frac{\norm{Z_s^d}_{\RR^d}}R\right)^{br}$, and the second bound follows from the first one with a larger moment exponent; the case $b=0$ is immediate. Finally, for $s<T$, the density of $Z_s^d$ is the convolution of the density $\mathbf 1_{[0,1]^d}$ with a Gaussian probability density, and therefore has $L^\infty$-norm at most one. For $s=T$, the density is $\mathbf 1_{[0,1]^d}$ itself.
\end{proof}

\begin{lemma}\label{lem:kolmogorov-g-approx}
    There exist constants $C,\alpha\in(0,\infty)$, depending only on $T,p,\rho,C_0,\kappa$, and $m$, such that for every $d\in\NN$, $R\geq1$, and $\eta\in(0,1]$, there exists a continuous finite expression $\widetilde g_{d,R,\eta}\colon\RR^d\to\RR$ satisfying
    \begin{equation}\label{eq:kolmogorov-g-compact-error}
        \norm{g_d-\widetilde g_{d,R,\eta}}_{L^p([-R,R]^d)}\leq\eta,
    \end{equation}
    $\abs{\widetilde g_{d,R,\eta}(x)} \leq 2C_0d^\kappa\left(1+ \left(\sum_{i=1}^d x_i^2\right)^{\ell/2}\right)$ for $x\in\RR^d$, and $\Cost(\widetilde g_{d,R,\eta}) \leq Cd^\alpha\br{1+d\log(2R)+\log(\eta^{-1})}^\alpha$. In addition,
    \begin{equation}\label{eq:kolmogorov-g-law-error}
        \norm{(g_d-\widetilde g_{d,R,\eta})(Z_0^d)}_{L^p(\Omega)}
        \leq \eta+Cd^\alpha R^{-2}.
    \end{equation}
\end{lemma}

\begin{proof}
    Set $A_{d,R}\coloneqq C_0d^\kappa \left(1+d^{m/2}R^m+(2R)^{\kappa+\rho}\right)$ and define $G_{d,R}(z) \coloneqq A_{d,R}^{-1}g_d(2Rz-R\mathbf 1_d)$ for $z\in[0,1]^d$. The assumptions in \eqref{eq:kolmogorov-data-assumptions} imply that $\abs{G_{d,R}}\leq1$ and $\abs{G_{d,R}(z)-G_{d,R}(w)}\leq\norm{z-w}_{\RR^d}^{\rho}$. Thus $G_{d,R}\in\calH_1^\rho([0,1]^d)$.
    Let
    \begin{equation}\label{eq:kolmogorov-normalized-tolerance}
        \delta\coloneqq \eta\br{A_{d,R}(2R)^{d/p}}^{-1}.
    \end{equation}
    By \cite[Theorem~5]{liang2025finiteexpressionmethodsolving}, there exists a continuous finite expression $A_{d,R,\eta}$ satisfying $\norm{G_{d,R}-A_{d,R,\eta}}_{L^p([0,1]^d)}\leq\delta$ and
    \begin{equation}\label{eq:kolmogorov-normalized-cost}
        \Cost(A_{d,R,\eta})
        \leq Cd^\alpha\br{1+\log d+\log(\delta^{-1})}^\alpha.
    \end{equation}
    Recall the coordinatewise projection $\Pi_d\colon\RR^d\to[0,1]^d$ from Section~\ref{sec:cost-conventions}, and set $a_{d,R,\eta}(x) \coloneqq\allowbreak A_{d,R}A_{d,R,\eta}\allowbreak \left(\Pi_d\left(\frac{x+R\mathbf 1_d}{2R}\right)\right)$. For $x\in[-R,R]^d$, the rescaled argument belongs to $[0,1]^d$, so the projection leaves it unchanged. By \eqref{eq:kolmogorov-normalized-tolerance}, the change of variables $x=2Rz-R\mathbf 1_d$ gives $\norm{g_d-a_{d,R,\eta}}_{L^p([-R,R]^d)}\leq\eta$. Put $E_d(x)\coloneqq\allowbreak 2C_0d^\kappa\allowbreak \left(1+\left(\sum_{i=1}^d x_i^2\right)^{\ell/2}\right)$ and define $\widetilde g_{d,R,\eta}(x) \coloneqq \Clip_{E_d}(x,a_{d,R,\eta}(x))$. Since $1+r^m\leq2(1+r^\ell)$ for $r\geq0$, the growth assumption gives $\abs{g_d(x)}\leq E_d(x)$. Consequently, clipping does not increase the error on the box. This proves \eqref{eq:kolmogorov-g-compact-error} and the global growth bound. The stated cost estimate follows from \eqref{eq:kolmogorov-normalized-cost}, the identity $\log(\delta^{-1}) =\log(\eta^{-1})+\log A_{d,R}+\frac d p\log(2R)$ given by \eqref{eq:kolmogorov-normalized-tolerance}, and the $O(d)$ cost of the affine rescaling, coordinatewise projection, and clipping.

    To prove \eqref{eq:kolmogorov-g-law-error}, split according to whether $Z_0^d\in[-R,R]^d$. The density bound in Lemma~\ref{lem:kolmogorov-gaussian-law} and \eqref{eq:kolmogorov-g-compact-error} control the part inside the box by $\eta$. On the complement, both $g_d$ and $\widetilde g_{d,R,\eta}$ have polynomial growth of order $\ell$, and Lemma~\ref{lem:kolmogorov-gaussian-law} with $b=2$ gives the remaining term.
\end{proof}

\begin{lemma}\label{lem:kolmogorov-f-approx}
    There exists $C\in(0,\infty)$, depending only on $L$, such that for every $\delta\in(0,1]$, there exists a continuous finite expression $\widetilde f_\delta\colon\RR\to\RR$ satisfying $\abs{\widetilde f_\delta(y)-\widetilde f_\delta(z)}\leq L\abs{y-z}$ for $y,z\in\RR$,
    \begin{equation}\label{eq:kolmogorov-f-error}
        \abs{f(y)-\widetilde f_\delta(y)}
        \leq\delta\max\{1,\abs y^2\}
        \quad\text{for}\quad y\in\RR,
    \end{equation}
    $\abs{\widetilde f_\delta(0)}\leq\abs{f(0)}+1$ and $\Cost(\widetilde f_\delta)\leq C(1+\delta^{-2})$.
\end{lemma}

\begin{proof}
    If $L=0$, then $f$ is constant and we take $\widetilde f_\delta=f$. Assume $L>0$ and set $S\coloneqq\max\left\{1,\frac{2L}{\delta}\right\}$ and $N\coloneqq\left\lceil\frac{4LS}{\delta}\right\rceil$. Let $\widetilde f_\delta$ be the piecewise linear interpolation of $f$ on the uniform grid of $[-S,S]$ with $N$ subintervals, extended constantly outside this interval. The slopes of the interpolant have absolute value at most $L$, and its interpolation error on $[-S,S]$ is at most $\delta/2$. Outside this interval, $\abs{f(y)-\widetilde f_\delta(y)} \leq L\max\{0,\abs y-S\} \leq\frac\delta2\abs y^2$. This proves \eqref{eq:kolmogorov-f-error}. The interpolant is a finite expression with evaluation cost $O(N)=O(1+\delta^{-2})$. The bound at zero follows from \eqref{eq:kolmogorov-f-error}.
\end{proof}

\subsection{Stability and moment estimates}\label{subsec:kolmogorov-stability}

\begin{lemma}\label{lem:kolmogorov-replacement}
    Let $u_d$ and $\widetilde u_d$ be continuous mild solutions with initial data $g_d$ and $\widetilde g_d$ and scalar nonlinearities $f$ and $\widetilde f$, respectively. Assume that $f$ is Lipschitz continuous with Lipschitz constant $L$ and that the $L^p$-norms appearing below are finite, with the corresponding norm functions integrable over the time intervals on which they are used. Then
    \begin{equation}\label{eq:kolmogorov-stability}
    \begin{aligned}
        \norm{u_d(T,\cdot)-\widetilde u_d(T,\cdot)}_{L^p([0,1]^d)}
        \leq e^{LT}\left(
        \norm{(g_d-\widetilde g_d)(Z_0^d)}_{L^p(\Omega)}
        +\int_0^T
        \norm{(f-\widetilde f)(\widetilde u_d(s,Z_s^d))}_{L^p(\Omega)}\dd s
        \right).
    \end{aligned}
    \end{equation}
\end{lemma}

\begin{proof}
    Let $\EE_W$ denote expectation with respect to the Brownian motion $W^d$, with $(Y_d,B^d)$ held fixed. For $t\in[0,T]$, the mild representations give
    \[
        u_d(t,Z_t^d)
        =\EE_W\left[g_d(Z_t^d+\sqrt2W_t^d)
        +\int_0^t f\br{u_d(s,Z_t^d+\sqrt2W_{t-s}^d)}\dd s\right],
    \]
    and the analogous identity for $\widetilde u_d$. Define $E(t)\coloneqq\norm{u_d(t,Z_t^d)-\widetilde u_d(t,Z_t^d)}_{L^p(\Omega)}$. Conditional Jensen's inequality and Minkowski's integral inequality, together with
    \begin{equation}\label{eq:kolmogorov-state-laws}
        Z_t^d+\sqrt2W_t^d\overset{\mathrm d}=Z_0^d
        \quad\text{and}\quad
        Z_t^d+\sqrt2W_{t-s}^d\overset{\mathrm d}=Z_s^d,
    \end{equation}
    yield
    \[
        E(t)\leq\norm{(g_d-\widetilde g_d)(Z_0^d)}_{L^p(\Omega)}
        +\int_0^t\norm{f(u_d(s,Z_s^d))-\widetilde f(\widetilde u_d(s,Z_s^d))}_{L^p(\Omega)}\dd s.
    \]
    Splitting the last difference and using the Lipschitz property of $f$ gives the Gr\"onwall inequality corresponding to \eqref{eq:kolmogorov-stability}. Since $Z_T^d=Y_d$, one has $E(T)=\norm{u_d(T,\cdot)-\widetilde u_d(T,\cdot)}_{L^p([0,1]^d)}$, and the result follows.
\end{proof}

\begin{lemma}\label{lem:kolmogorov-moment}
    Let $\widetilde g_{d,R,\eta}$ and $\widetilde f_\delta$ be as in Lemmas~\ref{lem:kolmogorov-g-approx} and \ref{lem:kolmogorov-f-approx}, and let $\widetilde u_{d,R,\eta,\delta}$ be the corresponding unique continuous mild solution of at most polynomial growth. There exist constants $C,\alpha\in(0,\infty)$, independent of $d,R,\eta$, and $\delta$, such that
    \begin{equation}\label{eq:kolmogorov-moment}
        \int_0^T
        \norm{\max\{1,\abs{\widetilde u_{d,R,\eta,\delta}(s,Z_s^d)}^2\}}_{L^p(\Omega)}\dd s
        \leq Cd^\alpha.
    \end{equation}
\end{lemma}

\begin{proof}
    The solution exists by the result cited in Section~\ref{sec:kolmogorov}. Its polynomial growth and Gaussian moment bounds justify the finiteness and time integrability of the quantities below. Set $r=2p$ and $M(t)\coloneqq\norm{\widetilde u_{d,R,\eta,\delta}(t,Z_t^d)}_{L^r(\Omega)}$. The conditional mild representation and the distributional identities \eqref{eq:kolmogorov-state-laws} give $M(t)\leq\norm{\widetilde g_{d,R,\eta}(Z_0^d)}_{L^r(\Omega)} +t\abs{\widetilde f_\delta(0)}+L\int_0^tM(s)\dd s$. The global polynomial bound for $\widetilde g_{d,R,\eta}$, Lemma~\ref{lem:kolmogorov-gaussian-law}, and the bound $\abs{\widetilde f_\delta(0)}\leq\abs{f(0)}+1$ show that the first two terms are bounded by $Cd^\alpha$, uniformly in the approximation parameters. Gr\"onwall's inequality gives $\sup_{t\in[0,T]}M(t)\leq Cd^\alpha$. Since $\norm{\max\{1,\abs X^2\}}_{L^p(\Omega)}\leq1+\norm{X}_{L^{2p}(\Omega)}^2$, integration over $[0,T]$ proves \eqref{eq:kolmogorov-moment}.
\end{proof}

\subsection{Finite expression realization of the MLP approximation}\label{subsec:kolmogorov-mlp}

\begin{lemma}\label{lem:kolmogorov-mlp}
    Fix $\calD\in\{\calD_0,\calD_\sigma\}$ and let $K,\beta\in[1,\infty)$. For every $d\in\NN$, let $h_d\colon\RR^d\to\RR$ and $\phi_d\colon\RR\to\RR$ be continuous finite expressions over $\calD$, and let $v_d$ be the unique continuous mild solution of at most polynomial growth of $\pdx{}{t}v_d(t,x)=\Delta_xv_d(t,x)+\phi_d(v_d(t,x))$ with $v_d(0,x)=h_d(x)$. Assume that, for every $d\in\NN$, $x\in\RR^d$, and $y,z\in\RR$, $\abs{\phi_d(y)-\phi_d(z)}\leq L\abs{y-z}$ and $\max\{\abs{\phi_d(0)},\abs{h_d(x)}\} \leq Kd^\beta\left(1+\sum_{i=1}^d\abs{x_i}\right)^\beta$. Suppose, in addition, that $\Cost(h_d)+\Cost(\phi_d)\leq A d^\beta$, where $A\in[1,\infty)$ is independent of $d$. Then there exist constants $C,\gamma\in(0,\infty)$, depending only on $T,p,L,K$, and $\beta$, but not on $A,d$, or $\eta$, such that for every $d\in\NN$ and every $\eta\in(0,1]$, there exists a deterministic finite expression $\Phi_{d,\eta}\colon\RR^d\to\RR$ over the same $\calD$ satisfying
    \begin{equation}\label{eq:kolmogorov-mlp-error}
        \norm{v_d(T,\cdot)-\Phi_{d,\eta}}_{L^p([0,1]^d)}\leq\eta
    \end{equation}
    and $\Cost(\Phi_{d,\eta})\leq C A d^\gamma\eta^{-\gamma}$.
\end{lemma}

\begin{proof}
Fix $d\in\NN$ and $\eta\in(0,1]$. Throughout the approximation estimates, $x\in[0,1]^d$ and $z\in[0,1/\sqrt2]^d$; $y\in\RR^d$ denotes an unrestricted spatial point. Constants $C$ depend only on $T,p,L,K,\beta$, as allowed by Section~\ref{sec:cost-conventions}.

\medskip\noindent
{\it Step 1. Setup.}
Existence and uniqueness of $v_d$ follow from \cite[Theorem~1.1]{Beck_2021}, as in Section~\ref{sec:kolmogorov}, using time reversal, zero drift, and diffusion matrix $\sqrt2 I_d$. Define, for $t\in[0,T]$ and $y\in\RR^d$, $w_d(t,y)\coloneqq v_d(T-t,\sqrt2y)$ and $\widehat h_d(y)\coloneqq h_d(\sqrt2y)$. Let $W$ be an auxiliary standard $d$-dimensional Brownian motion with continuous sample paths. Write $\EE_W$ for expectation over $W$ and $\norm{X}_{L^r(W)}\coloneqq\br{\EE_W[\abs{X}^r]}^{1/r}$ for $r\geq1$. Time reversal and spatial rescaling in the mild equation give
\begin{equation}\label{eq:kolmogorov-mlp-backward}
    w_d(t,y)=\EE_W\sbr{\widehat h_d(y+W_{T-t})}
    +\int_t^T\EE_W\sbr{\phi_d\br{w_d(s,y+W_{s-t})}}\dd s.
\end{equation}
In particular, $w_d(T,y)=\widehat h_d(y)$ and $w_d(0,x/\sqrt2)=v_d(T,x)$.

In \cite[Settings~2.1 and~3.1]{Hutzenthaler_2020}, take $u=w_d$, $g=\widehat h_d$, and $(Fq)(t,y)=\phi_d(q(t,y))$. Continuity and the $L$-Lipschitz property of $\phi_d$ give the required mapping and pointwise Lipschitz properties of $F$. The polynomial growth of $\widehat h_d,w_d$, uniform in time, together with Gaussian moments and $\abs{\phi_d(a)}\leq\abs{\phi_d(0)}+L\abs a$, verifies all the integrability conditions in Setting~3.1 for every deterministic reference point. Fubini's theorem identifies \eqref{eq:kolmogorov-mlp-backward} with its fixed-point equation.

\medskip\noindent
{\it Step 2. MLP.}
Use the independent Brownian motions and uniform random variables from Section~\ref{sec:kolmogorov}, independently of $W$, to define $U_{n,M}^{d,\theta}$ by the recursion in \cite[Setting~3.1]{Hutzenthaler_2020} with data $\widehat h_d,F$ and zero initial levels. This recursion does not depend on the reference point $\xi$. The same fields are therefore used as $z$ varies; they are jointly measurable with continuous sample functions by \cite[Lemma~3.2(i)]{Hutzenthaler_2020} and its construction. Write $\EE_\omega$ for expectation over this MLP sampling.

Since $(F0)(t,y)=\phi_d(0)$, the source's seminorm satisfies $\norm{F(0)}_1^2=\frac1T\int_0^T\EE_W\sbr{\abs{\phi_d(0)}^2}\dd s =\abs{\phi_d(0)}^2$. Thus \cite[Theorem~3.5]{Hutzenthaler_2020}, with $\xi=z$, gives, for every $n,M\in\NN$,
\begin{equation}\label{eq:kolmogorov-mlp-pointwise}
\begin{aligned}
    &\br{\EE_\omega\sbr{\abs{w_d(0,z)-U_{n,M}^{d,0}(0,z,\omega)}^2}}^{1/2}
    \leq e^{LT}\br{\norm{\widehat h_d(z+W_T)}_{L^2(W)}+T\abs{\phi_d(0)}}
    \frac{e^{M/2}(1+2LT)^n}{M^{n/2}}.
\end{aligned}
\end{equation}

\medskip\noindent
{\it Step 3. A uniform factor bound.}
Writing $W_T^{(i)}$ for the $i$th Brownian coordinate, Minkowski's inequality and $\abs{z_i}\leq1/\sqrt2$ give
\[
\begin{aligned}
    \norm{1+\sqrt2\sum_{i=1}^d\abs{z_i+W_T^{(i)}}}_{L^{2\beta}(W)}
    \leq1+d+\sqrt2d\norm{W_T^{(1)}}_{L^{2\beta}(W)}
    \leq Cd,
\end{aligned}
\]
since $W_T^{(1)}$ has law $\mathcal N(0,T)$. The growth assumption therefore yields
\[
    \norm{\widehat h_d(z+W_T)}_{L^2(W)}
    \leq Kd^\beta\norm{1+\sqrt2\sum_{i=1}^d\abs{z_i+W_T^{(i)}}}_{L^{2\beta}(W)}^\beta
    \leq Cd^{2\beta}.
\]
Since also $\abs{\phi_d(0)}\leq Kd^\beta$, choose $C_*\geq1$, depending only on $T,L,K,\beta$, and set $c_d\coloneqq C_*d^{2\beta}$ so that
\begin{equation}\label{eq:kolmogorov-mlp-prefactor}
    e^{LT}\br{\norm{\widehat h_d(z+W_T)}_{L^2(W)}+T\abs{\phi_d(0)}}\leq c_d
\end{equation}
uniformly on the stated $z$-cube. Applying \cite[Lemma~3.4]{Hutzenthaler_2020} with $\xi=z$ at time zero, where $W_0=0$, gives $\abs{w_d(0,z)}\leq c_d$. Hence
\begin{equation}\label{eq:kolmogorov-mlp-target-bound}
    \abs{v_d(T,x)}=\abs{w_d(0,x/\sqrt2)}\leq c_d.
\end{equation}

\medskip\noindent
{\it Step 4. Choosing a level and integrating error.}
Set $p_*\coloneqq\max\{2,p\}$ and
\begin{equation}\label{eq:kolmogorov-mlp-l2-tolerance}
    \eta_2\coloneqq\eta^{p_*/2}(2c_d)^{-(p_*-2)/2}\in(0,1].
\end{equation}
Choose
\begin{equation}\label{eq:kolmogorov-mlp-level}
    N\coloneqq\min\left\{n\in\NN:\ n\geq2,\ 
    c_d\frac{e^{n/2}(1+2LT)^n}{n^{n/2}}\leq\eta_2\right\},
\end{equation}
which is well defined by Lemma~\ref{lem:heat-mlp-level}. Equations~\eqref{eq:kolmogorov-mlp-pointwise}--\eqref{eq:kolmogorov-mlp-prefactor}, applied with $n=M=N$ and $z=x/\sqrt2$, imply $\EE_\omega\sbr{\abs{v_d(T,x)-U_{N,N}^{d,0}(0,x/\sqrt2,\omega)}^2}\leq\eta_2^2$. Joint measurability, Tonelli's theorem, and $\abs{[0,1]^d}=1$ now give
\begin{equation}\label{eq:kolmogorov-mlp-spatial-l2}
\begin{aligned}
    \EE_\omega\sbr{\int_{[0,1]^d}\abs{v_d(T,x)-U_{N,N}^{d,0}(0,x/\sqrt2,\omega)}^2\dd x}
    =\int_{[0,1]^d}\EE_\omega\sbr{\abs{v_d(T,x)-U_{N,N}^{d,0}(0,x/\sqrt2,\omega)}^2}\dd x
    \leq\eta_2^2.
\end{aligned}
\end{equation}

\medskip\noindent
{\it Step 5. Clipping and realization.}
The projection property in Section~\ref{sec:cost-conventions} gives, for $a\in[-c_d,c_d]$ and $b\in\RR$, $\abs{a-\pi_{c_d}(b)}^{p_*}\leq(2c_d)^{p_*-2}\abs{a-b}^2$. Using \eqref{eq:kolmogorov-mlp-target-bound}, \eqref{eq:kolmogorov-mlp-spatial-l2}, and \eqref{eq:kolmogorov-mlp-l2-tolerance}, we obtain
\begin{equation}\label{eq:kolmogorov-mlp-clipped-mean}
\begin{aligned}
    \EE_\omega\sbr{\int_{[0,1]^d}
    \abs{v_d(T,x)-\pi_{c_d}\br{U_{N,N}^{d,0}(0,x/\sqrt2,\omega)}}^{p_*}\dd x}
    \leq(2c_d)^{p_*-2}\eta_2^2=\eta^{p_*}.
\end{aligned}
\end{equation}
Choose an outcome $\omega_{d,\eta}$ for which the integral in \eqref{eq:kolmogorov-mlp-clipped-mean} is at most $\eta^{p_*}$, and define on $\RR^d$
\begin{equation}\label{eq:kolmogorov-mlp-realization}
    \Phi_{d,\eta}(y)\coloneqq\pi_{c_d}\br{U_{N,N}^{d,0}(0,y/\sqrt2,\omega_{d,\eta})}.
\end{equation}
Since $p\leq p_*$ and the unit cube has volume one, $\norm{v_d(T,\cdot)-\Phi_{d,\eta}}_{L^p([0,1]^d)} \leq\norm{v_d(T,\cdot)-\Phi_{d,\eta}}_{L^{p_*}([0,1]^d)}\leq\eta$, which proves \eqref{eq:kolmogorov-mlp-error}.

\medskip\noindent
{\it Step 6. Evaluation cost.}
Since $\Cost(\widehat h_d)\leq\Cost(h_d)+d$, the recurrence argument of Lemma~\ref{lem:fullcost}, with $C_g=\Cost(\widehat h_d)$ and $C_f=\Cost(\phi_d)$, bounds one randomized recursive evaluation by
\[
    \FullCost\br{U_{n,M}^{d,\theta}(t,y)}
    \leq C\br{d+\Cost(\widehat h_d)+\Cost(\phi_d)}(5M)^n
    \leq CAd^\beta(5M)^n.
\]
The Brownian increments in this recursion have the same conditional Gaussian sampling cost as in Subsection~\ref{subsec:cost}. At fixed time zero and outcome $\omega_{d,\eta}$, all recursive times and displacements are constants with respect to the spatial input. The closure rules in Section~\ref{sec:cost-conventions} therefore show that \eqref{eq:kolmogorov-mlp-realization} is over $\calD$ and, including its final input scaling and clipping,
\begin{equation}\label{eq:kolmogorov-mlp-frozen-cost}
    \Cost(\Phi_{d,\eta})\leq CAd^\beta(5N)^N+Cd.
\end{equation}
Lemma~\ref{lem:heat-mlp-level}, applied to \eqref{eq:kolmogorov-mlp-level} with $\overline c=c_d$, accuracy $\eta_2$, and $\delta=1$, gives $(5N)^N\leq C(1+c_d^3)\eta_2^{-3}\leq Cd^{6\beta}\eta_2^{-3}$. Substituting this and \eqref{eq:kolmogorov-mlp-l2-tolerance} into \eqref{eq:kolmogorov-mlp-frozen-cost} yields
\begin{equation}\label{eq:kolmogorov-mlp-final-cost}
\begin{aligned}
    \Cost(\Phi_{d,\eta})
    \leq CAd^{7\beta}(2c_d)^{3(p_*-2)/2}\eta^{-3p_*/2}+Cd
    \leq CAd^{\beta(3p_*+1)}\eta^{-3p_*/2}.
\end{aligned}
\end{equation}
Taking $\gamma=\beta(3p_*+1)\geq3p_*/2$ in \eqref{eq:kolmogorov-mlp-final-cost} proves the cost assertion. The choices $c_d,\eta_2,N$ are independent of $A$, so the bound is linear in $A$.
\end{proof}

\begin{proof}[Proof of Theorem~\ref{thm:kolmogorov-main}]
    Fix $\e\in(0,1]$, and let $C_1,\alpha_1\in(0,\infty)$ be constants supplied by Lemma~\ref{lem:kolmogorov-g-approx}, independent of $d$ and $\e$. Then, for every $d\in\NN$, $R\geq1$, and $\eta\in(0,1]$, $\norm{(g_d-\widetilde g_{d,R,\eta})(Z_0^d)}_{L^p(\Omega)} \leq\eta+C_1d^{\alpha_1}R^{-2}$. Set $\eta_g\coloneqq\frac{\e e^{-LT}}8$ and $R\coloneqq\max\left\{1, \left(\frac{8C_1e^{LT}d^{\alpha_1}}\e\right)^{1/2}\right\}$, and choose $\widetilde g_d\coloneqq\widetilde g_{d,R,\eta_g}$. Then
    \begin{equation}\label{eq:kolmogorov-final-g-error}
        \norm{(g_d-\widetilde g_d)(Z_0^d)}_{L^p(\Omega)}\leq\frac{\e e^{-LT}}4.
    \end{equation}

    Lemma~\ref{lem:kolmogorov-moment} provides constants $C_2\in[1,\infty)$ and $\alpha_2\in(0,\infty)$ such that the bound in \eqref{eq:kolmogorov-moment} is at most $M_d\coloneqq C_2d^{\alpha_2}\geq1$ for every scalar approximation tolerance. Set $\eta_f\coloneqq\frac{\e e^{-LT}}{4M_d}$ and let $\widetilde f_d\coloneqq\widetilde f_{\eta_f}$ be given by Lemma~\ref{lem:kolmogorov-f-approx}. Let $\widetilde u_d$ be the unique continuous mild solution of at most polynomial growth with data $(\widetilde g_d,\widetilde f_d)$. By \eqref{eq:kolmogorov-f-error} and \eqref{eq:kolmogorov-moment}, $\int_0^T\norm{(f-\widetilde f_d)(\widetilde u_d(s,Z_s^d))}_{L^p(\Omega)}\dd s \leq\eta_fM_d=\frac{\e e^{-LT}}4$. Lemma~\ref{lem:kolmogorov-replacement} and \eqref{eq:kolmogorov-final-g-error} therefore give
    \begin{equation}\label{eq:kolmogorov-replacement-final}
        \norm{u_d(T,\cdot)-\widetilde u_d(T,\cdot)}_{L^p([0,1]^d)}\leq\frac\e2.
    \end{equation}

    The choices of $R$, $\eta_g$, and $\eta_f$, together with Lemmas~\ref{lem:kolmogorov-g-approx} and \ref{lem:kolmogorov-f-approx}, imply that there exist constants $C_3\in[1,\infty)$ and $\alpha_3\in(0,\infty)$ such that $\Cost(\widetilde g_d)+\Cost(\widetilde f_d) \leq C_3d^{\alpha_3}\e^{-\alpha_3}$. The polynomial growth bound for $\widetilde g_d$, the bound on $\widetilde f_d(0)$, and the Lipschitz constant of $\widetilde f_d$ are uniform in the approximation parameters. Set $\beta\coloneqq\max\{1,\kappa,\ell,\alpha_3\}$. After increasing a fixed constant $K$, the growth assumption in Lemma~\ref{lem:kolmogorov-mlp} holds with this $\beta$, while its cost assumption holds with $A=C_3\e^{-\alpha_3}$. Carrying out the same construction in every dimension at this fixed accuracy gives a family with the same constants $K$, $\beta$, and $A$. Apply Lemma~\ref{lem:kolmogorov-mlp} to this family with $\calD=\calD_0$ and MLP accuracy $\e/2$. This gives a finite expression $\Psi_{d,\e}$ satisfying $\norm{\widetilde u_d(T,\cdot)-\Psi_{d,\e}}_{L^p([0,1]^d)}\leq\frac\e2$ and $\Cost(\Psi_{d,\e}) \leq CAd^\gamma(\e/2)^{-\gamma} \leq C'd^\gamma\e^{-(\alpha_3+\gamma)}$. Increasing the constants and choosing $\alpha\geq\alpha_3+\gamma$ gives $\Cost(\Psi_{d,\e})\leq Cd^\alpha\e^{-\alpha}$. Combining the two spatial error bounds proves the theorem.
\end{proof}

\section{Approximation through stochastic representations}\label{sec:stochastic-transfer}
Suppose a PDE solution, at a fixed time when applicable, has the representation $u(x)=\EE[g(S(x,\zeta))]$. We replace $g$ by a finite expression, bound the resulting spatial $L^p$-error, and approximate the expectation by a finite average. Sections~\ref{sec:black-scholes} and~\ref{sec:elliptic-halfspace} apply these estimates to linear PDEs.

We use Borel $\sigma$-algebras and their products. Let $d,k,r\in\NN$, $p\in[1,\infty)$, and let $D\subseteq\RR^d$, $G\subseteq\RR^k$, and $E\subseteq\RR^r$ be measurable, with $0<|D|<\infty$. Here $D$ is the approximation region, and $|D|$ and spatial $L^p$-norms use Lebesgue measure. Let $(\Omega,\calF,\PP)$ be a probability space carrying an $E$-valued random variable $\zeta$ and a random variable $Y$ uniformly distributed on $D$, independent of $\zeta$. Write $\nu$ for the law of $\zeta$. Let $S\colon D\times E\to G$ be jointly measurable. For any measurable $h\colon G\to\RR$ and $x\in D$, we consider the evaluation operator $(T_xh)(z)\coloneqq h(S(x,z))$ for $z\in E$ and define $(\calP h)(x)\coloneqq\int_E(T_xh)(z)\,\nu(\dd z)=\EE[h(S(x,\zeta))]$ whenever this integral converges absolutely, and set it to be $0$ otherwise. The integrability hypotheses below ensure that this exception occurs on a set of Lebesgue measure $0$. We denote the induced state and law by $Z\coloneqq S(Y,\zeta)$ and $\mu(B)\coloneqq\PP(Z\in B)$ for any measurable $B\subseteq G$.

\begin{theorem}\label{thm:stochastic-transfer}
    In the setting above, let $g,\widetilde g\colon G\to\RR$ be measurable with $\EE\sbr{|g(Z)|^p+|\widetilde g(Z)|^p}\allowbreak<\infty$. Let $u\coloneqq\calP g$ and $\widetilde u\coloneqq\calP\widetilde g$. Let $Q\subseteq G$ be compact and $A>0$ so that the restriction of $\mu$ to $Q$ has density $\rho_Q$ satisfying
    $\mu(B)=\int_B\rho_Q(z)\dd z$ for every measurable $B\subseteq Q$ and $0\leq\rho_Q\leq A$ almost everywhere in $Q$. Suppose $\eta\in[0,\infty)$ satisfies $\norm{g-\widetilde g}_{L^p(Q)}\leq\eta$.
    Define the tail error outside of $Q$ to be $\tau_Q\coloneqq\norm{(g-\widetilde g)(Z)\mathbf 1_{\{Z\notin Q\}}}_{L^p(\Omega)}$. Then $u,\widetilde u\in L^p(D)$, $\norm{(g-\widetilde g)(Z)}_{L^p(\Omega)}\leq (A\eta^p+\tau_Q^p)^{1/p}\leq A^{1/p}\eta+\tau_Q$, and
    \begin{equation}\label{eq:stochastic-transfer}
        \norm{u-\widetilde u}_{L^p(D)}\leq|D|^{1/p}(A^{1/p}\eta+\tau_Q).
    \end{equation}
    If, in addition, $|\widetilde g|\leq K$ on $G$ for some $K\in[0,\infty)$, then the tail error satisfies
    \begin{equation}\label{eq:stochastic-transfer-tail}
        \tau_Q\leq\norm{g(Z)\mathbf 1_{\{Z\notin Q\}}}_{L^p(\Omega)}+K\PP(Z\notin Q)^{1/p}.
    \end{equation}
\end{theorem}
\begin{proof}
    Set $F\coloneqq g-\widetilde g$. For $h\in\{g,\widetilde g,F\}$, the integrability assumptions on $g$ and $\widetilde g$, together with independence of $Y$ and $\zeta$, give, by Tonelli's theorem,
    \begin{equation}\label{eq:transfer-product-law}
        \int_D\int_E|h(S(x,z))|^p\,\nu(\dd z)\dd x
        =|D|\EE[|h(Z)|^p]<\infty.
    \end{equation}
    Since $\nu(E)=1$, this ensures that $\calP h$ is measurable and its defining integral converges absolutely for almost every $x\in D$. Jensen's inequality and \eqref{eq:transfer-product-law} then give
    \begin{equation}\label{eq:transfer-averaging-bound}
        \norm{\calP h}_{L^p(D)}^p\leq|D|\EE[|h(Z)|^p].
    \end{equation}
    Therefore $u,\widetilde u\in L^p(D)$ and $u-\widetilde u=\calP F$ almost everywhere. The density bound on $Q$ gives
    \begin{equation}\label{eq:transfer-compact-tail}
        \EE[|F(Z)|^p]
        =\int_Q|F(z)|^p\rho_Q(z)\dd z+\tau_Q^p
        \leq A\norm{F}_{L^p(Q)}^p+\tau_Q^p
        \leq A\eta^p+\tau_Q^p.
    \end{equation}
    Combining \eqref{eq:transfer-averaging-bound} and \eqref{eq:transfer-compact-tail} with $h=F$ and taking $p$\th roots gives the asserted bounds, including \eqref{eq:stochastic-transfer}, since $(a^p+b^p)^{1/p}\leq a+b$ for $a,b\geq0$.
    Finally, if $|\widetilde g|\leq K$, then $\norm{\widetilde g(Z)\mathbf 1_{\{Z\notin Q\}}}_{L^p(\Omega)}\leq K\PP(Z\notin Q)^{1/p}$, so the triangle inequality gives \eqref{eq:stochastic-transfer-tail}, which completes the proof.
\end{proof}
\begin{lemma}\label{lem:finite-average-realization}
    Assume the setting of Section~\ref{sec:stochastic-transfer}. Let $K\in[0,\infty)$ and $h\colon G\to\RR$ be measurable with $\abs{h}\leq K$ on $G$. There exists $C_p\in(0,\infty)$, depending only on $p$, such that for every $n\in\NN$, there are $z^1,\ldots,z^n\in E$ for which the function $\Psi_n(x)\coloneqq\frac{1}{n}\sum_{j=1}^nh(S(x,z^j))$, defined for $x\in D$, satisfies
    \begin{equation}\label{eq:finite-average-error}
        \norm{\calP h-\Psi_n}_{L^p(D)}\leq C_p\abs{D}^{1/p}Kn^{-1/2}.
    \end{equation}
    The same samples $z^1,\ldots,z^n$ are used for all $x\in D$ in the above bound. Fix $\calD\in\{\calD_0,\calD_\sigma\}$ and suppose, additionally, that $h$ is the restriction to $G$ of a globally defined finite expression $\widehat h\colon\RR^k\to\RR$ over $\calD$. Assume that there exist a measurable set $E_0\subseteq E$ with $\nu(E_0)=1$ and a constant $C_S\in[0,\infty)$ such that for every $z\in E_0$, the map $S(\cdot,z)$ admits a globally defined finite expression extension $\widehat S_z\colon\RR^d\to\RR^k$ over $\calD$ satisfying $\widehat S_z(x)=S(x,z)$ for $x\in D$ and $\Cost(\widehat S_z)\leq C_S$. Then the samples can be chosen in $E_0$, and $\Psi_n$ admits a globally defined finite expression extension over $\calD$ given by $\widehat\Psi_n(x)\coloneqq\frac{1}{n}\sum_{j=1}^n\widehat h(\widehat S_{z^j}(x))$ for $x\in\RR^d$, with cost satisfying
    \begin{equation}\label{eq:finite-average-cost}
        \Cost(\widehat\Psi_n)\leq n\br{\Cost(\widehat h)+C_S+1}.
    \end{equation}
\end{lemma}
\begin{proof}
Fix $n\in\NN$ and let $\zeta^1,\ldots,\zeta^n$ be independent copies of $\zeta$ on an auxiliary probability space. Write $\EE_n$ for expectation on this space and define $\Psi_n^\omega(x)\coloneqq\frac{1}{n}\sum_{j=1}^nh(S(x,\zeta^j(\omega)))$ for $x\in D$ and $\xi_j(x)\coloneqq h(S(x,\zeta^j))-(\calP h)(x)$ for $j=1,\ldots,n$. Since $\abs{(\calP h)(x)}\leq K$, for each fixed $x\in D$, the random variables $\xi_j(x)$ are independent, centered, and bounded in absolute value by $2K$. Let $q=2\lceil p/2\rceil$. In the expansion
\[
    \EE_n\sbr{\abs{\sum_{j=1}^n\xi_j(x)}^q}
    =\sum_{i_1,\ldots,i_q=1}^n\EE_n\sbr{\xi_{i_1}(x)\cdots\xi_{i_q}(x)},
\]
a summand vanishes if any index appears exactly once, by independence and the zero-mean property. Thus for a summand to be nonzero, every distinct index must occur at least twice. Therefore, there can be at most $q/2$ distinct indices, so there are at most $C_qn^{q/2}$ such nonzero terms, where $C_q$ is a constant depending only on $q$. Since each term is bounded in absolute value by $(2K)^q$, we find, after increasing $C_q$ if necessary,
\begin{equation}\label{eq:finite-average-even-moment}
    \EE_n\sbr{\abs{\sum_{j=1}^n\xi_j(x)}^q}\leq C_qK^qn^{q/2}.
\end{equation}
As $p\leq q$, it follows from Jensen's inequality and \eqref{eq:finite-average-even-moment} that
\begin{equation}\label{eq:finite-average-pointwise}
\begin{aligned}
    \EE_n\sbr{\abs{(\calP h)(x)-\Psi_n^\omega(x)}^p}
    \leq\br{\EE_n\sbr{\abs{\frac{1}{n}\sum_{j=1}^n\xi_j(x)}^q}}^{p/q}
    \leq C_p^pK^pn^{-p/2},
\end{aligned}
\end{equation}
where $C_p$ depends only on $p$. Joint measurability, Fubini's theorem, and \eqref{eq:finite-average-pointwise} yield
\[
\begin{aligned}
    \EE_n\sbr{\norm{\calP h-\Psi_n^\omega}_{L^p(D)}^p}
    =\int_D\EE_n\sbr{\abs{\Psi_n^\omega(x)-(\calP h)(x)}^p}\dd x
    \leq C_p^p\abs{D}K^pn^{-p/2}.
\end{aligned}
\]
Consequently, there are realizations $z^1,\ldots,z^n\in E$ for which \eqref{eq:finite-average-error} holds.

Under the additional assumptions, the tuple $(z^1,\ldots,z^n)$ can be chosen in $E_0^n$, since $\nu^{\otimes n}(E_0^n)=1$, and we may define $\widehat\Psi_n(x)\coloneqq\frac{1}{n}\sum_{j=1}^n\widehat h(\widehat S_{z^j}(x))$ for $x\in\RR^d$. This is a globally defined finite expression over $\calD$. For $x\in D$, one has $\widehat h(\widehat S_{z^j}(x))=h(S(x,z^j))$, so $\widehat\Psi_n|_D=\Psi_n$. Each summand requires at most $\Cost(\widehat h)+C_S$ operations; summing the $n$ values requires $n-1$ additions, and multiplying by $1/n$ requires one more operation. Hence $\Cost(\widehat\Psi_n)\leq n\br{\Cost(\widehat h)+C_S+1}$, which gives \eqref{eq:finite-average-cost} and completes the proof.
\end{proof}

\section{Black--Scholes equations}\label{sec:black-scholes}
We treat diagonal Black--Scholes equations with constant drift and volatility parameters. We approximate the payoff on boxes and use the estimates of Section~\ref{sec:stochastic-transfer} to control the error under the lognormal terminal distribution. A selected Monte Carlo average then gives a finite expression approximation.

Let $T>0$, $p\in[1,\infty)$, $\rho\in(0,1]$, and let $L,C_0,m,\kappa\in[1,\infty)$. For each $d\in\NN$, let $\alpha_{d,i},\beta_{d,i}\in\RR$, $i=1,\ldots,d$, satisfy
$\max_{1\leq i\leq d}\left\{\abs{\alpha_{d,i}}+\abs{\beta_{d,i}}\right\}\leq L$.
Let $\f_d\colon[0,\infty)^d\to\RR$ be continuous. We consider the Black--Scholes equation
\[
\pdx{u_d}{t}(t,x)
+\sum_{i=1}^d \alpha_{d,i}x_i\pdx{u_d}{x_i}(t,x)
+\frac12\sum_{i=1}^d\beta_{d,i}^2x_i^2\frac{\partial^2u_d}{\partial x_i^2}(t,x)=0,
\]
for $(t,x)\in[0,T)\times(0,\infty)^d$, with terminal condition
$u_d(T,x)=\f_d(x)$ for $x\in[0,\infty)^d$.
We assume that for all $d\in\NN$ and $x\in[0,\infty)^d$,
$\abs{\f_d(x)}\leq C_0d^{\kappa}\br{1+\norm{x}_{\RR^d}^m}$,
and that for all $d\in\NN$, $R\geq1$, and $x,y\in[0,R]^d$,
$\abs{\f_d(x)-\f_d(y)}\leq C_0d^{\kappa}R^{\kappa}\norm{x-y}_{\RR^d}^{\rho}$.
Let $(\Omega,\calF,\PP)$ be a probability space carrying, for every $d\in\NN$, a standard $d$-dimensional Brownian motion $W^d=(W_1^d,\ldots,W_d^d)$; in particular, its coordinate processes are independent standard one-dimensional Brownian motions. For $t\in[0,T]$, set
$M_{d,i}^{t,T}:=\exp\br{\br{\alpha_{d,i}-\frac12\beta_{d,i}^2}(T-t)+\beta_{d,i}\br{W_i^d(T)-W_i^d(t)}}$,
with
$M_d^{t,T}:=\allowbreak(M_{d,1}^{t,T},\ldots,M_{d,d}^{t,T})$ and $M_d:=M_d^{0,T}$.
For $x\in[0,\infty)^d$, write
$x\odot M_d^{t,T}=(x_1M_{d,1}^{t,T},\ldots,x_dM_{d,d}^{t,T})$.
The polynomial growth assumption and the finite moments of the lognormal multipliers imply that the following expectations are finite. Define
$u_d(t,x):=\EE\sbr{\f_d(x\odot M_d^{t,T})}$ for $(t,x)\in[0,T]\times[0,\infty)^d$.
After extending $\f_d$ to $\RR^d$ by coordinatewise positive-part projection, \cite[Proposition~2.23]{Grohs_2023}, applied after time reversal, shows that $u_d$ is the restriction to $[0,T]\times[0,\infty)^d$ of the unique continuous viscosity solution of at most polynomial growth of the corresponding full-space equation. In particular,
$u_d(0,x)=\EE\sbr{\f_d(x\odot M_d)}$ for $x\in[0,1]^d$.
Finally, throughout this section, $Y_d$ denotes a random variable uniformly distributed on $[0,1]^d$, independent of $W^d$, and
$Z_d\coloneqq Y_d\odot M_d$.
For each fixed $d$, we use the setting and measurability conventions of Section~\ref{sec:stochastic-transfer} with $k=r=d$ and
$D=[0,1]^d$, $G=[0,\infty)^d$, $E=(0,\infty)^d$, $Y=Y_d$, and $\zeta=M_d$.
The map $S(x,z)=x\odot z$ is jointly continuous from $D\times E$ to $G$, $Y_d$ is independent of $M_d$, and $\abs{D}=1$. Thus $Z=Z_d$ and $u_d(0,x)=(\calP\f_d)(x)$ for $x\in D$.

\begin{theorem}\label{thm:bs-fex}
Under the assumptions in Section~\ref{sec:black-scholes}, there exist constants $C,\gamma\in(0,\infty)$, depending only on $p,T,\rho,L,C_0,m$, and $\kappa$, such that for every $d\in\NN$ and $\e\in(0,1]$, there exists a finite expression $\Psi_{d,\e}\colon\RR^d\to\RR$ over $\calD_0$ satisfying
$\norm{u_d(0,\cdot)-\Psi_{d,\e}}_{L^p([0,1]^d)}\leq \e$ and $\Cost(\Psi_{d,\e})\leq Cd^{\gamma}\e^{-\gamma}$.
In particular, in the unit-cost evaluation model of Section~\ref{sec:cost-conventions}, the finite expression approximation of the Black--Scholes solution on $[0,1]^d$ avoids the curse of dimensionality.
\end{theorem}

\subsection{Stochastic estimates, data approximation, and realization}\label{subsec:black-scholes-lemmas}

\begin{lemma}\label{lem:bs-moment-tail}
For every $r,b\in[0,\infty)$ there exist $C,\gamma\in(0,\infty)$, depending only on $r,b,p,T,L,m$, such that for all $d\in\NN$ and $R\geq1$, we have $\EE\sbr{\norm{Z_d}_{\RR^d}^r}\leq Cd^{\gamma}$ and $\norm{(1+\norm{Z_d}_{\RR^d}^m)\mathbf 1_{Z_d\notin[0,R]^d}}_{L^p(\Omega)}\leq Cd^{\gamma}R^{-b}$.
\end{lemma}
\begin{proof}
Fix $a\in[0,\infty)$ and $i\in\{1,\ldots,d\}$. Since $W_i^d(T)\sim N(0,T)$, for every $\lambda\in\RR$,
$\EE[\exp(\lambda W_i^d(T))]=\exp\br{\frac12\lambda^2T}$.
Hence
$\EE\sbr{M_{d,i}^a} =\exp\br{a\alpha_{d,i}T+\frac12(a^2-a)\beta_{d,i}^2T} \leq \exp\br{aLT+\frac12(a^2+a)L^2T}$.
The right-hand side is independent of $d$ and $i$. Since $(Y_d)_i\in[0,1]$,
$\norm{Z_d}_{\RR^d}\leq \sum_{i=1}^d M_{d,i}$.
If $r\geq1$, convexity gives
$\EE\sbr{\norm{Z_d}_{\RR^d}^r} \leq d^{r-1}\sum_{i=1}^d\EE\sbr{M_{d,i}^r} \leq Cd^r$.
If $r\in[0,1)$, by the concavity of $x\mapsto x^r$ and Jensen's inequality,
$\EE\sbr{\norm{Z_d}_{\RR^d}^r} \leq \br{\EE\sbr{\norm{Z_d}_{\RR^d}}}^r \leq Cd^r$.
This proves the first claim.
For the second assertion, if $b=0$, then the first claim, applied with $r=mp$, gives
$\norm{(1+\norm{Z_d}_{\RR^d}^m)\mathbf 1_{Z_d\notin[0,R]^d}}_{L^p(\Omega)} \leq \norm{1+\norm{Z_d}_{\RR^d}^m}_{L^p(\Omega)} \leq Cd^{\gamma}$.
Hence assume $b>0$. Since $Z_d\in[0,\infty)^d$ almost surely, the event $\{Z_d\notin[0,R]^d\}$ implies $\norm{Z_d}_{\RR^d}>R$. Therefore,
$\mathbf 1_{Z_d\notin[0,R]^d} \leq \br{\frac{\norm{Z_d}_{\RR^d}}{R}}^{bp}$.
Consequently,
$\norm{(1+\norm{Z_d}_{\RR^d}^m)\mathbf 1_{Z_d\notin[0,R]^d}}_{L^p(\Omega)}^p \leq R^{-bp}\EE\sbr{(1+\norm{Z_d}_{\RR^d}^m)^p\norm{Z_d}_{\RR^d}^{bp}}$.
For $x\geq0$,
$(1+x^m)^px^{bp}\leq C(1+x^{p(m+b)})$.
Using the first claim with $r=p(m+b)$ gives
$\norm{(1+\norm{Z_d}_{\RR^d}^m)\mathbf 1_{Z_d\notin[0,R]^d}}_{L^p(\Omega)}^p \leq Cd^{\gamma p}R^{-bp}$.
Taking $p$\th roots finishes the proof.
\end{proof}

\begin{lemma}\label{lem:bs-density-transfer}
For every $d\in\NN$, $Z_d$ has a density $\rho_d$ on $\RR^d$ satisfying
$\norm{\rho_d}_{L^\infty(\RR^d)}\leq\exp\br{T(L+L^2)d}$.
\end{lemma}

\begin{proof}
	Recall that the law of $Y_d$ is the usual Lebesgue measure on $[0,1]^d$. Since $Y_d$ and $M_d$ are independent, their joint law is the product of their respective laws. For every nonnegative measurable $F\colon G\to[0,\infty)$ and $m\in E$, Tonelli's theorem and the change of variables $z=y\odot m$ give
\[
\begin{aligned}
\EE\sbr{F(Z_d)}
=\EE\sbr{\int_{[0,1]^d}F(y\odot M_d)\dd y}
=\EE\sbr{\br{\prod_{i=1}^dM_{d,i}^{-1}}
\int_{\prod_{i=1}^d[0,M_{d,i}]}F(z)\dd z}
=\int_G F(z)\rho_d(z)\dd z,
\end{aligned}
\]
where
$\rho_d(z)\coloneqq \EE\sbr{\br{\prod_{i=1}^dM_{d,i}\inv}\mathbf 1_{\prod_{i=1}^d[0,M_{d,i}]}(z)}$
for $z\in\RR^d$. So $\rho_d$ is the density of $Z_d$, zero outside $G$. Moreover, the definition and assumptions on $M_{d,i}\inv$ give
$\EE\sbr{M_{d,i}^{-1}} =\exp\br{-\alpha_{d,i}T+\beta_{d,i}^2T} \leq \exp\br{LT+L^2T}$.
The independence of the Brownian coordinates shows that $M_{d,1},\ldots,M_{d,d}$ are also independent. Therefore, for every $z\in\RR^d$,
$0\leq\rho_d(z) \leq\EE\sbr{\prod_{i=1}^dM_{d,i}^{-1}} =\prod_{i=1}^d\EE\sbr{M_{d,i}^{-1}} \leq\exp\br{T(L+L^2)d}$,
which finishes the proof.
\end{proof}

\begin{lemma}\label{lem:bs-compact-fex}
Under the standing growth and local H\"older assumptions on $\f_d$, there exist $C,\gamma\in(0,\infty)$, depending only on $p,\rho,C_0,m$, and $\kappa$, such that for every $d\in\NN$, $R\geq1$, and $\eta\in(0,1]$ there exists a finite expression $a_{d,R,\eta}\colon\RR^d\to\RR$ so that $\norm{\f_d-a_{d,R,\eta}}_{L^p([0,R]^d)}\leq \eta$ and $\Cost(a_{d,R,\eta})\leq C d^{\gamma}\br{1+d\log(2R)+\log(\eta^{-1})}^{\gamma}$.
\end{lemma}

\begin{proof}
Let
$B_{d,R}:=C_0d^{\kappa}\br{1+d^{m/2}R^m+R^{\kappa+\rho}}$
and define
$F(z):=B_{d,R}^{-1}\f_d(Rz)$ for $z\in[0,1]^d$.
Then $F\in\calH_1^\rho([0,1]^d)$. Indeed, for $z\in[0,1]^d$,
$\abs{F(z)} \leq B_{d,R}^{-1}C_0d^{\kappa}\br{1+d^{m/2}R^m} \leq1$,
and for $z,w\in[0,1]^d$,
$\abs{F(z)-F(w)} \leq B_{d,R}^{-1}C_0d^{\kappa}R^{\kappa+\rho}\norm{z-w}_{\RR^d}^\rho \leq \norm{z-w}_{\RR^d}^\rho$.
Set
\begin{equation}\label{eq:bs-normalized-tolerance}
\delta_{d,R,\eta}:=\frac{\eta}{B_{d,R}R^{d/p}}\in(0,1].
\end{equation}
By \cite[Theorem~5]{liang2025finiteexpressionmethodsolving}, there is a finite expression $A_{d,R,\eta}$ satisfying
$\norm{F-A_{d,R,\eta}}_{L^p([0,1]^d)} \leq \delta_{d,R,\eta}$
and
\begin{equation}\label{eq:bs-source-cost}
\Cost(A_{d,R,\eta})
\leq
Cd^2\br{1+\log d+\log(\delta_{d,R,\eta}^{-1})}^2.
\end{equation}
With the coordinatewise projection $\Pi_d\colon\RR^d\to[0,1]^d$ from Section~\ref{sec:cost-conventions}, set
$a_{d,R,\eta}(x):=\allowbreak B_{d,R}A_{d,R,\eta}(\Pi_d(x/R))$ for $x\in\RR^d$.
For $x\in[0,R]^d$, one has $\Pi_d(x/R)=x/R$. Hence, by the change of variables $x=Rz$ and \eqref{eq:bs-normalized-tolerance},
$\norm{\f_d-a_{d,R,\eta}}_{L^p([0,R]^d)} = B_{d,R}R^{d/p}\norm{F-A_{d,R,\eta}}_{L^p([0,1]^d)} \leq \eta$.
Moreover, \eqref{eq:bs-normalized-tolerance} gives
$\log(\delta_{d,R,\eta}^{-1}) = \log B_{d,R}+\frac{d}{p}\log R+\log(\eta^{-1}) \leq C\br{1+d\log(2R)+\log(\eta^{-1})+\log d}$.
Since $d\geq1$ and $R\geq1$,
$\log d\leq d\log 2\leq d\log(2R)$,
so the $\log d$ term is absorbed by $d\log(2R)$. Evaluation of $a_{d,R,\eta}$ requires at most $O(d)$ scalar operations in addition to evaluating $A_{d,R,\eta}$. Combining this with \eqref{eq:bs-source-cost} completes the proof.
\end{proof}

\begin{lemma}\label{lem:bs-payoff-approx}
There exist $C,\gamma\in(0,\infty)$, depending only on $p,T,L,C_0,m,\rho$, and $\kappa$, such that for every $d\in\NN$, $R\geq1$, and $\eta\in(0,1]$, there exists a finite expression $\widetilde\f_{d,R,\eta}\colon\RR^d\to\RR$ satisfying
\begin{enumerate}[label=(\roman*)]
\item $\norm{\br{\f_d-\widetilde\f_{d,R,\eta}}(Z_d)}_{L^p(\Omega)}\leq e^{Cd}\eta+Cd^{\gamma}R^{-2}$.
\item $\Cost(\widetilde\f_{d,R,\eta})\leq Cd^{\gamma}\br{1+d\log(2R)+\log(\eta^{-1})}^{\gamma}$.
\end{enumerate}
Moreover,
$\abs{\widetilde\f_{d,R,\eta}(x)}\leq K_{d,R}:=C_0d^{\kappa}\br{1+d^{m/2}R^m}$ for $x\in\RR^d$.
\end{lemma}

\begin{proof}
Let $a_{d,R,\eta}$ be the approximation from Lemma~\ref{lem:bs-compact-fex}. The growth assumption on $\f_d$ gives
$\abs{\f_d(x)}\leq K_{d,R}$ for $x\in[0,R]^d$.
Set
$\widetilde\f_{d,R,\eta}(x):=\pi_{K_{d,R}}(a_{d,R,\eta}(x))$.
The clipping properties in Section~\ref{sec:cost-conventions} give a finite expression bounded globally by $K_{d,R}$, preserve the cost bound from Lemma~\ref{lem:bs-compact-fex}, and yield
\begin{equation}\label{eq:bs-clipped-compact-error}
\norm{\f_d-\widetilde\f_{d,R,\eta}}_{L^p([0,R]^d)}\leq\eta.
\end{equation}
Lemma~\ref{lem:bs-moment-tail}, the growth bound on $\f_d$, and the boundedness of $\widetilde\f_{d,R,\eta}$ give
$\EE\sbr{\abs{\f_d(Z_d)}^p+\abs{\widetilde\f_{d,R,\eta}(Z_d)}^p}<\infty$.
Set $Q=[0,R]^d$ and $A=\exp\br{T(L+L^2)d}$. Lemma~\ref{lem:bs-density-transfer} and \eqref{eq:bs-clipped-compact-error} obey the density and error hypotheses of Theorem~\ref{thm:stochastic-transfer}. Apply its induced-law and tail bounds with $g=\f_d$, $\widetilde g=\widetilde\f_{d,R,\eta}|_G$, and $K=K_{d,R}$. Since $A^{1/p}=\exp\br{T(L+L^2)d/p}\leq e^{Cd}$ for $C\geq T(L+L^2)/p$, these give
\begin{equation}\label{eq:bs-payoff-error-split}
\begin{aligned}
\norm{\br{\f_d-\widetilde\f_{d,R,\eta}}(Z_d)}_{L^p(\Omega)}
&\leq e^{Cd}\eta
+\norm{\f_d(Z_d)\mathbf 1_{Z_d\notin[0,R]^d}}_{L^p(\Omega)}+K_{d,R}\PP(Z_d\notin[0,R]^d)^{1/p}.
\end{aligned}
\end{equation}
For the tail, the growth condition and Lemma~\ref{lem:bs-moment-tail}, with $b=2$, imply
$\norm{\f_d(Z_d)\mathbf 1_{Z_d\notin[0,R]^d}}_{L^p(\Omega)} \leq Cd^{\gamma}R^{-2}$.
Since $\mathbf 1_{Z_d\notin[0,R]^d}\leq(1+\norm{Z_d}_{\RR^d}^m)\mathbf 1_{Z_d\notin[0,R]^d}$, Lemma~\ref{lem:bs-moment-tail}, with $b=m+2$, yields
$\PP(Z_d\notin[0,R]^d)^{1/p} \leq Cd^{\gamma}R^{-(m+2)}$.
Therefore
$K_{d,R}\PP(Z_d\notin[0,R]^d)^{1/p} \leq Cd^{\gamma}\br{1+R^m}R^{-(m+2)} \leq Cd^{\gamma}R^{-2}$.
Combining these tail estimates with \eqref{eq:bs-payoff-error-split} proves part (i).
\end{proof}

\begin{lemma}\label{lem:bs-mc-fex}
Fix $\calD\in\{\calD_0,\calD_\sigma\}$. Let $K\in[0,\infty)$ and let $g\colon\RR^d\to\RR$ be a finite expression over $\calD$ with $\abs{g}\leq K$. Then for every $n\in\NN$ there exist $m^1,\ldots,m^n\in(0,\infty)^d$ such that
\[
\norm{\EE\sbr{g(\cdot\odot M_d)}-\frac1n\sum_{j=1}^ng(\cdot\odot m^j)}_{L^p([0,1]^d)}\leq C_pKn^{-1/2}.
\]
Moreover, the finite average
$x\mapsto \frac1n\sum_{j=1}^ng(x\odot m^j)$
is a finite expression over $\calD$ of cost at most $Cn(\Cost(g)+d)$.
\end{lemma}

\begin{proof}
Apply Lemma~\ref{lem:finite-average-realization} with the same dictionary $\calD$, the setting in Section~\ref{sec:black-scholes}, $h=g|_G$, $\widehat h=g$, and $E_0=E$. The function $h$ is measurable and bounded by $K$. For each $z\in E$, the map
$\widehat S_z(x)=x\odot z$, $x\in\RR^d$,
is a globally defined finite expression over $\calD$ extending $S(\cdot,z)$ with cost at most $d$, so take $C_S=d$. Since $\nu(E_0)=1$ and $\abs{D}=1$, \eqref{eq:finite-average-error} and \eqref{eq:finite-average-cost} give the stated bounds, with $m^j=z^j$ and $d+1\leq2d$.
\end{proof}

\begin{proof}[Proof of Theorem~\ref{thm:bs-fex}]
Let $d\in\NN$ and $\e\in(0,1]$. Throughout this proof, constants denoted by $C$ and exponents denoted by $\gamma$ may change from line to line, but depend only on $p,T,L,C_0,m,\rho$, and $\kappa$. Let $R\geq1$ and $\eta\in(0,1]$, let $\widetilde\f_{d,R,\eta}$ be as in Lemma~\ref{lem:bs-payoff-approx}, and define
$\widetilde u_{d,R,\eta}(t,x):=\EE\sbr{\widetilde\f_{d,R,\eta}(x\odot M_d^{t,T})}$ for $(t,x)\in[0,T]\times[0,\infty)^d$.
On $D$, $\widetilde u_{d,R,\eta}(0,\cdot)=\calP(\widetilde\f_{d,R,\eta}|_G)$. The integrability checked in the proof of Lemma~\ref{lem:bs-payoff-approx} and \eqref{eq:transfer-averaging-bound}, applied to $h=\f_d-\widetilde\f_{d,R,\eta}|_G$, give, since $\abs{D}=1$, $\norm{u_d(0,\cdot)-\widetilde u_{d,R,\eta}(0,\cdot)}_{L^p([0,1]^d)}
\leq
\norm{\br{\f_d-\widetilde\f_{d,R,\eta}}(Z_d)}_{L^p(\Omega)}$. By Lemma~\ref{lem:bs-payoff-approx}, there exist $A_1,\gamma_1\in[1,\infty)$, depending only on $p,T,L,C_0,m,\rho$, and $\kappa$, such that
$\norm{u_d(0,\cdot)-\widetilde u_{d,R,\eta}(0,\cdot)}_{L^p([0,1]^d)} \leq e^{A_1d}\eta+A_1d^{\gamma_1}R^{-2}$.
Choose
\begin{equation}\label{eq:bs-parameters}
R:=\max\left\{1,\br{\frac{4A_1d^{\gamma_1}}{\e}}^{1/2}\right\}
\quad\text{and}\quad
\eta:=\frac{e^{-A_1d}\e}{4}.
\end{equation}
Then $\eta\in(0,1]$ and
\begin{equation}\label{eq:bs-data-error}
\norm{u_d(0,\cdot)-\widetilde u_{d,R,\eta}(0,\cdot)}_{L^p([0,1]^d)}\leq\frac\e2.
\end{equation}
Moreover,
$R\leq Cd^\gamma\e^{-\gamma}$.
The bound in Lemma~\ref{lem:bs-payoff-approx} gives
$\abs{\widetilde\f_{d,R,\eta}(x)}
\leq K_{d,R}=C_0d^{\kappa}\br{1+d^{m/2}R^m}$ for $x\in\RR^d$. The preceding bound on $R$ implies
$K_{d,R}\leq Cd^\gamma\e^{-\gamma}$.
Apply Lemma~\ref{lem:bs-mc-fex} with $\calD=\calD_0$, $g=\widetilde\f_{d,R,\eta}$, and $K=K_{d,R}$. Choose
\begin{equation}\label{eq:bs-sample-size}
n:=\left\lceil4C_p^2K_{d,R}^2\e^{-2}\right\rceil.
\end{equation}
Then $n\leq Cd^\gamma\e^{-\gamma}$, and Lemma~\ref{lem:bs-mc-fex} gives points $m^1,\ldots,m^n\in(0,\infty)^d$ such that $\Psi_{d,\e}(x):=\frac1n\sum_{j=1}^n\widetilde\f_{d,R,\eta}(x\odot m^j)$ satisfies $\norm{\widetilde u_{d,R,\eta}(0,\cdot)-\Psi_{d,\e}}_{L^p([0,1]^d)}
\leq C_pK_{d,R}n^{-1/2}
\leq \frac\e2$. The triangle inequality, \eqref{eq:bs-data-error}, and the preceding estimate prove
$\norm{u_d(0,\cdot)-\Psi_{d,\e}}_{L^p([0,1]^d)}\leq\e$. It remains to bound the cost. Lemma~\ref{lem:bs-payoff-approx} gives $\Cost(\widetilde\f_{d,R,\eta})
\leq
Cd^\gamma\br{1+d\log(2R)+\log(\eta^{-1})}^{\gamma}$. For $R$ and $\eta$ from \eqref{eq:bs-parameters}, $\log(2R)\leq\allowbreak C\br{1+\log d+\log(\e^{-1})}$ and $\log(\eta^{-1})=A_1d+\log4+\log(\e^{-1})$. Although $\eta$ is exponentially small in $d$, its contribution to the evaluation-cost bound in Lemma~\ref{lem:bs-payoff-approx}(ii) is through $\log(\eta^{-1})$. Thus the exponential density-transfer factor contributes only a term proportional to $d$ inside the cost polynomial and does not prevent the final polynomial bound. Consequently,
$1+d\log(2R)+\log(\eta^{-1}) \leq C\br{1+d\log d+d\log(\e^{-1})+d} \leq Cd^2\e^{-1}$,
where we used $\log(\e^{-1})\leq\e^{-1}$, $\log d\leq d$, $d\geq1$, and $\e\in(0,1]$. Thus
$\Cost(\widetilde\f_{d,R,\eta})\leq Cd^\gamma\e^{-\gamma}$.
Finally, Lemma~\ref{lem:bs-mc-fex}, the choice of $n$ in \eqref{eq:bs-sample-size}, and the preceding bounds give
$\Cost(\Psi_{d,\e}) \leq Cn\br{\Cost(\widetilde\f_{d,R,\eta})+d} \leq Cd^\gamma\e^{-\gamma}$.
This completes the proof of Theorem~\ref{thm:bs-fex}.
\end{proof}

\section{Elliptic Dirichlet problems on the half-space}\label{sec:elliptic-halfspace}
We approximate the Poisson extension of bounded boundary data to a half-space. The proof combines finite expression approximation of the data with the estimates of Section~\ref{sec:stochastic-transfer}. Here the sampling distribution has Cauchy tails, so boundedness is used to control the tail error.

For $d\geq2$, let
$H_d\coloneqq \RR^{d-1}\times(0,\infty)$ and $\partial H_d=\RR^{d-1}\times\{0\}$.
We consider the Dirichlet problem
$\Delta u_d=0$ in $H_d$ with boundary condition $\lim_{x_d\downarrow0}u_d(x',x_d)=g_d(x')$ for $x'\in\RR^{d-1}$.
Fix $\kappa\in(0,1)$, independently of $d$, and define the interior slab $D_{d,\kappa}\coloneqq [-1/2,1/2]^{d-1}\times[\kappa,1]\subset H_d$. The approximation region $D_{d,\kappa}$ stays at distance at least $\kappa$ from the boundary.
Let $p\in[1,\infty)$, $\rho\in(0,1]$, $C_g\in[1,\infty)$, and $\beta,\gamma\in[0,\infty)$. For each $d\geq2$, let $g_d\colon\RR^{d-1}\to\RR$ be continuous. Assume that, for every $d\geq2$ and $z\in\RR^{d-1}$,
$\abs{g_d(z)}\leq C_gd^\beta$,
and that, for every $d\geq2$, $R\geq1$, and $z_1,z_2\in[-R,R]^{d-1}$,
$\abs{g_d(z_1)-g_d(z_2)} \leq C_gd^\beta R^\gamma\norm{z_1-z_2}_{\RR^{d-1}}^{\rho}$.
Boundedness also ensures that the Poisson extension is well defined and controls the Monte Carlo error.
Set
$a_d\coloneqq \frac{\Gamma(d/2)}{\pi^{d/2}}$
and define    $q_d(z)=a_d\br{1+\norm{z}_{\RR^{d-1}}^2}^{-d/2}$ for $z\in\RR^{d-1}$. The standard integral formula for powers of $1+\norm{z}_{\RR^{d-1}}^2$ gives
    $\int_{\RR^{d-1}}q_d(z)\dd z
    =
    a_d\pi^{(d-1)/2}\frac{\Gamma(1/2)}{\Gamma(d/2)}
    =1$, so $q_d$ is a probability density. Let $C_d$ be an $\RR^{d-1}$-valued random variable with density $q_d$.
The Poisson extension of $g_d$ to $H_d$ is $u_d(x',x_d)
    \coloneqq
    \EE\sbr{g_d(x'+x_dC_d)}$ for $(x',x_d)\in H_d$. Equivalently,
\[
    u_d(x',x_d)
    =
    \int_{\RR^{d-1}}
    g_d(y)a_d\frac{x_d}{\br{\norm{x'-y}_{\RR^{d-1}}^2+x_d^2}^{d/2}}\dd y.
\]
The change of variables $y=x'+x_dz$ gives the Poisson kernel in the second formula. By \cite[Section~2.2, Theorem~14]{evans2010partial}, $u_d$ belongs to $C^\infty(H_d)\cap L^\infty(H_d)$, is harmonic in $H_d$, and satisfies the stated boundary condition.
For each fixed $d\geq2$, we use the setting and measurability conventions of Section~\ref{sec:stochastic-transfer} with $k=r=d-1$ and
$D=D_{d,\kappa}$, $G=E=\RR^{d-1}$, $\zeta=C_d$.
Let $Y=(Y',Y_d)$ be uniformly distributed on $D$ and independent of $C_d$. The map
$S((x',x_d),z)=x'+x_dz$
is jointly continuous from $D\times E$ to $G$, and $0<\abs{D}=1-\kappa<1$. Thus $Z=S(Y,\zeta)=Y'+Y_dC_d=Z_d$ and, for bounded measurable $h\colon G\to\RR$,
$(\calP h)(x',x_d)=\EE\sbr{h(x'+x_dC_d)}$ for $(x',x_d)\in D$.
In particular, $u_d|_D=\calP g_d$.

\begin{theorem}\label{thm:elliptic-halfspace-fex}
Let $p\in[1,\infty)$, $\rho\in(0,1]$, $\kappa\in(0,1)$, $C_g\in[1,\infty)$, and $\beta,\gamma\in[0,\infty)$. Assume that $g_d\colon\RR^{d-1}\to\RR$, $d\geq2$, satisfy the boundedness and local H\"older assumptions stated in Section~\ref{sec:elliptic-halfspace}. Let $u_d$ be the Poisson extension of $g_d$ to $H_d$. Then there exist constants $C,\alpha\in(0,\infty)$, depending only on $p,\rho,\kappa,C_g,\beta,\gamma$, such that for every $d\geq2$ and every $\e\in(0,1]$, there exists a finite expression $\Psi_{d,\e}\colon\RR^d\to\RR$ over $\calD_0$ satisfying
$\norm{u_d-\Psi_{d,\e}}_{L^p(D_{d,\kappa})}\leq \e$ and $\Cost(\Psi_{d,\e})\leq Cd^\alpha\e^{-\alpha}$.
In particular, the family of Poisson solutions $u_d$ admits finite expression approximations on $D_{d,\kappa}$ without the curse of dimensionality.
\end{theorem}
\subsection{Stochastic estimates, data approximation, and realization}\label{subsec:elliptic-lemmas}

\begin{lemma}\label{lem:elliptic-density-tail}
Let $d\geq2$, let $Y=(Y',Y_d)$ be uniformly distributed on $D_{d,\kappa}$, and assume that $Y$ is independent of $C_d$. Define
$Z_d\coloneqq Y'+Y_dC_d\in\RR^{d-1}$.
Then $Z_d$ has a density $r_d$ satisfying
$\norm{r_d}_{L^{\infty}(\RR^{d-1})}\leq1$.
Moreover, there exists a universal constant $C\in(0,\infty)$ such that for all $R\geq2$,
$\PP\br{Z_d\notin[-R,R]^{d-1}} \leq \frac{Cd}{R}$.
\end{lemma}
\begin{proof}
Set $Q_0\coloneqq[-1/2,1/2]^{d-1}$. Since $Y$ is uniform on $Q_0\times[\kappa,1]$ and independent of $C_d$, the random vector $Y'$ is uniform on $Q_0$ and independent of the pair $(Y_d,C_d)$, hence of $Y_dC_d$. Also, $\abs{Q_0}=1$. Thus, for every nonnegative measurable $f\colon\RR^{d-1}\to[0,\infty)$, independence, translation, and Tonelli's theorem give
\[
\begin{aligned}
    \EE\sbr{f(Z_d)}
    =\EE\sbr{\int_{Q_0}f(y+Y_dC_d)\dd y}
    =\EE\sbr{\int_{\RR^{d-1}}f(z)\mathbf 1_{Q_0}(z-Y_dC_d)\dd z}
    =\int_{\RR^{d-1}}f(z)\EE\sbr{\mathbf 1_{Q_0}(z-Y_dC_d)}\dd z.
\end{aligned}
\]
Consequently, $Z_d$ has the measurable density $r_d(z)\coloneqq\EE\sbr{\mathbf 1_{Q_0}(z-Y_dC_d)}$ for $z\in\RR^{d-1}$, which satisfies $0\leq r_d(z)\leq1$ for every $z\in\RR^{d-1}$.
It remains to prove the tail estimate. For $z\in\RR^{d-1}$, write
$\norm{z}_\infty\coloneqq\max_{1\leq i\leq d-1}\abs{z_i}$.
If $Z_d\notin[-R,R]^{d-1}$ and $R\geq2$, then at least one coordinate has absolute value larger than $R$. Since $Y'\in[-1/2,1/2]^{d-1}$, this implies
$\norm{Y_dC_d}_\infty\geq R-\frac12\geq \frac R2$.
Since $Y_d\leq1$, it follows that $\norm{C_d}_\infty\geq R/2$. Each coordinate of $C_d$ is a standard one-dimensional Cauchy random variable. To see this, write $n=d-1$ and integrate the density $\frac{\Gamma((n+1)/2)}{\pi^{(n+1)/2}}\br{1+c^2+\norm{w}_{\RR^{n-1}}^2}^{-(n+1)/2}$ over $w\in\RR^{n-1}$ using the standard identity
\[
    \int_{\RR^k}(A+\norm{w}_{\RR^k}^2)^{-\lambda}\dd w
    =
	\pi^{k/2}\frac{\Gamma(\lambda-k/2)}{\Gamma(\lambda)}A^{k/2-\lambda}\quad\text{for}\quad A>0\text{ and }\lambda>\frac{k}{2}.
\]
The same formula also holds for $k=0$, covering $d=2$. With $k=n-1$, $A=1+c^2$, and $\lambda=(n+1)/2$, it gives the marginal density $\pi^{-1}(1+c^2)^{-1}$. Therefore
$\PP\br{\abs{(C_d)_i}\geq a} \leq \frac{C}{a}$ for $a>0$.
The union bound gives
$\PP\br{Z_d\notin[-R,R]^{d-1}} \leq \PP\br{\norm{C_d}_\infty\geq R/2} \leq \sum_{i=1}^{d-1}\PP\br{\abs{(C_d)_i}\geq R/2} \leq \frac{Cd}{R}$.
\end{proof}

\begin{lemma}\label{lem:elliptic-boundary-fex}
There exist constants $C,\alpha\in(0,\infty)$, depending only on $p,\rho,C_g,\beta,\gamma$, such that for every $d\geq2$, $R\geq2$, and $\eta\in(0,1]$, there exists a finite expression $\widetilde g_{d,R,\eta}\colon\RR^{d-1}\to\RR$ satisfying $\norm{g_d-\widetilde g_{d,R,\eta}}_{L^p([-R,R]^{d-1})} \leq \eta$, $\abs{\widetilde g_{d,R,\eta}(z)}\leq C_gd^\beta$ for $z\in\RR^{d-1}$, and $\Cost(\widetilde g_{d,R,\eta}) \leq Cd^\alpha\br{1+d\log(2R)+\log(\eta^{-1})}^\alpha$.
\end{lemma}
\begin{proof}
Rescale $[-R,R]^{d-1}$ to the unit cube by setting
$x=2Rz-R\mathbf 1_{d-1}$ for $z\in[0,1]^{d-1}$.
Define
$B_{d,R}\coloneqq C_gd^\beta\br{1+(2R)^{\gamma+\rho}}$
and $F(z)\coloneqq B_{d,R}^{-1}g_d(2Rz-R\mathbf 1_{d-1})$ for $z\in[0,1]^{d-1}$. The boundedness assumption gives $\abs{F}\leq1$. If $z,w\in[0,1]^{d-1}$, then both $2Rz-R\mathbf 1_{d-1}$ and $2Rw-R\mathbf 1_{d-1}$ lie in $[-R,R]^{d-1}$, and the local H\"older assumption gives $\abs{F(z)-F(w)}
    \leq
    B_{d,R}^{-1}C_gd^\beta R^\gamma(2R)^\rho
    \norm{z-w}_{\RR^{d-1}}^\rho$. Since $R^\gamma(2R)^\rho\leq (2R)^{\gamma+\rho}$ for $R\geq1$, the factor in front of $\norm{z-w}_{\RR^{d-1}}^\rho$ is at most one. Thus $F\in\calH_1^\rho([0,1]^{d-1})$.
By \cite[Theorem~5]{liang2025finiteexpressionmethodsolving}, for every $\delta\in(0,1]$ there exists a finite expression $A_{d,R,\delta}$ such that
$\norm{F-A_{d,R,\delta}}_{L^p([0,1]^{d-1})} \leq\delta$ and
\begin{equation}\label{eq:elliptic-normalized-cost}
    \Cost(A_{d,R,\delta})
    \leq
    Cd^\alpha\br{1+\log d+
    \log(\delta^{-1})}^\alpha.
\end{equation}
Each operator in the expression constructed by the cited theorem has unit evaluation cost under the convention of Section~\ref{sec:cost-conventions}; hence its operator-count estimate yields the displayed evaluation-cost bound up to a universal multiplicative constant.
Choose
\begin{equation}\label{eq:elliptic-normalized-tolerance}
    \delta\coloneqq \eta\br{B_{d,R}(2R)^{(d-1)/p}}^{-1}.
\end{equation}
Then $\delta\in(0,1]$. Define $a_{d,R,\eta}(x)
    \coloneqq
    B_{d,R}A_{d,R,\delta}\br{\frac{x+R\mathbf 1_{d-1}}{2R}}$ for $x\in\RR^{d-1}$. For $x\in[-R,R]^{d-1}$, put $z=(x+R\mathbf 1_{d-1})/(2R)$. Since $\dd x=(2R)^{d-1}\dd z$, we have
\[
\begin{aligned}
    \norm{g_d-a_{d,R,\eta}}_{L^p([-R,R]^{d-1})}^p
    =
    B_{d,R}^p(2R)^{d-1}
    \norm{F-A_{d,R,\delta}}_{L^p([0,1]^{d-1})}^p
    \leq
    B_{d,R}^p(2R)^{d-1}\delta^p
    =
    \eta^p.
\end{aligned}
\]
The expression $a_{d,R,\eta}$ need not be globally bounded. Set $K_d\coloneqq C_gd^\beta$ and define
$\widetilde g_{d,R,\eta}\coloneqq \pi_{K_d}\circ a_{d,R,\eta}$.
Since $\abs{g_d}\leq K_d$, the clipping property in Section~\ref{sec:cost-conventions} preserves the compact error bound and gives $\abs{\widetilde g_{d,R,\eta}}\leq K_d$ globally. Finally, \eqref{eq:elliptic-normalized-tolerance} gives $\log(\delta^{-1})
    =
    \log(\eta^{-1})+\log B_{d,R}+\frac{d-1}{p}\log(2R)
    \leq
    C\br{1+d\log(2R)+\log(\eta^{-1})}$. Using \eqref{eq:elliptic-normalized-cost} and this estimate, we absorb the $O(d)$ cost of affine rescaling, multiplication by $B_{d,R}$, and clipping into the stated bound by choosing $\alpha\geq1$.
\end{proof}
\begin{lemma}\label{lem:elliptic-boundary-perturbation}
Let $d\geq2$, $R\geq2$, and $\eta\in(0,1]$, and let $\widetilde g_{d,R,\eta}$ be as in Lemma~\ref{lem:elliptic-boundary-fex}. Define
$\widetilde u_{d,R,\eta}(x',x_d) \coloneqq \EE\sbr{\widetilde g_{d,R,\eta}(x'+x_dC_d)}$.
Then there exist constants $C,\alpha\in(0,\infty)$, depending only on $p$, $C_g$, and $\beta$, such that
$\norm{u_d-\widetilde u_{d,R,\eta}}_{L^p(D_{d,\kappa})} \leq \eta+Cd^\alpha\br{\frac dR}^{1/p}$.
\end{lemma}
\begin{proof}
Use the setting in Section~\ref{sec:elliptic-halfspace} and set $Q=[-R,R]^{d-1}$. The functions $g_d$ and $\widetilde g_{d,R,\eta}$ are measurable and bounded in absolute value by $C_gd^\beta$, so $\EE\sbr{\abs{g_d(Z_d)}^p+\abs{\widetilde g_{d,R,\eta}(Z_d)}^p}
    \leq2(C_gd^\beta)^p<\infty$. Lemma~\ref{lem:elliptic-density-tail} supplies the density bound $r_d\leq1$, and Lemma~\ref{lem:elliptic-boundary-fex} gives $\norm{g_d-\widetilde g_{d,R,\eta}}_{L^p(Q)}\leq\eta$. On $D$, we have $u_d=\calP g_d$ and $\widetilde u_{d,R,\eta}=\calP\widetilde g_{d,R,\eta}$. Apply Theorem~\ref{thm:stochastic-transfer} with $g=g_d$, $\widetilde g=\widetilde g_{d,R,\eta}$, $A=1$, and $K=C_gd^\beta$. Equations~\eqref{eq:stochastic-transfer} and~\eqref{eq:stochastic-transfer-tail}, together with the tail bound in Lemma~\ref{lem:elliptic-density-tail}, give
$\norm{u_d-\widetilde u_{d,R,\eta}}_{L^p(D_{d,\kappa})} \leq(1-\kappa)^{1/p}\br{\eta+2C_gd^\beta\PP(Z_d\notin Q)^{1/p}} \leq\eta+Cd^\beta\br{\frac dR}^{1/p}$.
Here $C$ depends only on $p$ and $C_g$. Taking $\alpha=\max\{1,\beta\}$ proves the lemma.
\end{proof}

\begin{lemma}\label{lem:elliptic-mc}
Let $d\geq2$, $K\in[0,\infty)$, and let $h\colon\RR^{d-1}\to\RR$ be measurable and satisfy $\abs{h}\leq K$ on $\RR^{d-1}$. Then, for every $n\in\NN$, there exist points $c^1,\ldots,c^n\in\RR^{d-1}$ such that the function defined by $\Psi_n(x',x_d)\coloneqq \frac1n\sum_{j=1}^n h(x'+x_dc^j)$ for $(x',x_d)\in\RR^d$ satisfies $\norm{\calP h-\Psi_n}_{L^p(D_{d,\kappa})} \leq C_pKn^{-1/2}$, where $C_p\in(0,\infty)$ depends only on $p$. If, in addition, $h$ is a globally defined finite expression over $\calD\in\{\calD_0,\calD_\sigma\}$, then $\Psi_n$ is a globally defined finite expression over the same $\calD$ with $\Cost(\Psi_n)\leq Cn\br{\Cost(h)+d}$, where $C\in(0,\infty)$ is universal.
\end{lemma}
\begin{proof}
Apply Lemma~\ref{lem:finite-average-realization} with the setting in Section~\ref{sec:elliptic-halfspace}. Since $h$ is measurable and bounded by $K$ on $G$, \eqref{eq:finite-average-error} gives points $c^j=z^j$, $j=1,\ldots,n$, for which
$\norm{\calP h-\Psi_n}_{L^p(D_{d,\kappa})}
    \leq C_p(1-\kappa)^{1/p}Kn^{-1/2}
    \leq C_pKn^{-1/2}$. When $h$ is a globally defined finite expression over $\calD$, use the finite expression part of Lemma~\ref{lem:finite-average-realization} with the same $\calD$, $\widehat h=h$, and $E_0=E$. For every $z\in E$, the map
$\widehat S_z(x',x_d)=x'+x_dz$, $(x',x_d)\in\RR^d$,
is a globally defined finite expression over $\calD$ extending $S(\cdot,z)$ with total cost at most $2(d-1)$, which can be taken to be $C_S$. Since $\nu(E_0)=1$, the finite expression promised by Lemma~\ref{lem:finite-average-realization} is precisely what we are looking for, and \eqref{eq:finite-average-cost} shows
$\Cost(\Psi_n) \leq n\br{\Cost(h)+2(d-1)+1} \leq2n\br{\Cost(h)+d}$.
This proves the stated cost bound with $C=2$.
\end{proof}
\begin{proof}[Proof of Theorem~\ref{thm:elliptic-halfspace-fex}]
Fix $d\geq2$ and $\e\in(0,1]$. Let $C_1,\alpha_1\in(0,\infty)$ be constants such that Lemma~\ref{lem:elliptic-boundary-perturbation} gives $\norm{u_d-\widetilde u_{d,R,\eta}}_{L^p(D_{d,\kappa})}
    \leq
    \eta+C_1d^{\alpha_1}\br{\frac dR}^{1/p}$ for all $R\geq2$ and $\eta\in(0,1]$. Choose
\begin{equation}\label{eq:elliptic-radius-choice}
    R
    \coloneqq
    \left\lceil
    \max\left\{2,(4C_1)^pd^{p\alpha_1+1}\e^{-p}\right\}
    \right\rceil.
\end{equation}
Then $C_1d^{\alpha_1}\br{\frac dR}^{1/p}
    \leq
    C_1d^{\alpha_1}
    \br{\frac{d}{(4C_1)^pd^{p\alpha_1+1}\e^{-p}}}^{1/p}
    =
    \frac{\e}{4}$. Choose
$\eta\coloneqq\frac{\e}{4}$.
Then $\eta\in(0,1]$. Let $\widetilde g_{d,R,\eta}$ be given by Lemma~\ref{lem:elliptic-boundary-fex} and define $\widetilde u_{d,R,\eta}$ as in Lemma~\ref{lem:elliptic-boundary-perturbation}. The preceding choices give
\begin{equation}\label{eq:elliptic-perturbation-half-error}
    \norm{u_d-\widetilde u_{d,R,\eta}}_{L^p(D_{d,\kappa})}
    \leq
    \frac{\e}{2}.
\end{equation}
Set $K_d\coloneqq C_gd^\beta$ and choose
\begin{equation}\label{eq:elliptic-sample-count}
    n\coloneqq \left\lceil \br{\frac{2C_pK_d}{\e}}^2\right\rceil.
\end{equation}
By Lemma~\ref{lem:elliptic-boundary-fex}, $\widetilde g_{d,R,\eta}$ is a globally defined finite expression with $\abs{\widetilde g_{d,R,\eta}}\leq K_d$. Apply Lemma~\ref{lem:elliptic-mc} with $\calD=\calD_0$, this value of $n$, and $h=\widetilde g_{d,R,\eta}$. It gives deterministic points $c^1,\ldots,c^n\in\RR^{d-1}$ such that
\[
    \Psi_{d,\e}(x',x_d)
    \coloneqq
    \frac1n\sum_{j=1}^n\widetilde g_{d,R,\eta}(x'+x_dc^j)
\]
satisfies $\norm{\widetilde u_{d,R,\eta}-\Psi_{d,\e}}_{L^p(D_{d,\kappa})}
    \leq
    C_pK_dn^{-1/2}
    \leq
    \frac\e2$. The same lemma shows that $\Psi_{d,\e}$ is a globally defined finite expression. Together with \eqref{eq:elliptic-perturbation-half-error}, the triangle inequality gives $\norm{u_d-\Psi_{d,\e}}_{L^p(D_{d,\kappa})}
    \allowbreak\leq\allowbreak
    \norm{u_d-\widetilde u_{d,R,\eta}}_{L^p(D_{d,\kappa})}\allowbreak+\allowbreak
    \norm{\widetilde u_{d,R,\eta}-\Psi_{d,\e}}_{L^p(D_{d,\kappa})}
    \allowbreak\leq\allowbreak
    \e$.
It remains to estimate the cost. Constants denoted by $C\in(0,\infty)$ and $\alpha\in(0,\infty)$ may change from line to line, but depend only on $p$, $\rho$, $\kappa$, $C_g$, $\beta$, and $\gamma$. From the choice of $R$ in \eqref{eq:elliptic-radius-choice}, $R\leq Cd^\alpha\e^{-p}$ and $\log(2R)\leq C\br{1+\log d+\log(\e^{-1})}$. Moreover, the choice of $\eta$ gives
$\log(\eta^{-1})=\log4+\log(\e^{-1})$.
Lemma~\ref{lem:elliptic-boundary-fex} gives $\Cost(\widetilde g_{d,R,\eta})
    \leq
    Cd^\alpha\br{1+d\log(2R)+\log(\eta^{-1})}^\alpha$. The preceding bounds imply $1+d\log(2R)+\log(\eta^{-1})
    \leq
    C\br{1+d\log d+d\log(\e^{-1})}
    \leq
    Cd^2\e^{-1}$, where we used $\log d\leq d$ and $\log(\e^{-1})\leq\e^{-1}$. Hence,
$\Cost(\widetilde g_{d,R,\eta}) \leq Cd^\alpha\e^{-\alpha}$.
Also, \eqref{eq:elliptic-sample-count} gives
$n\leq Cd^\alpha\e^{-2}$.
Finally, the cost bound in Lemma~\ref{lem:elliptic-mc} gives $\Cost(\Psi_{d,\e})
    \leq
    Cn\br{d+\Cost(\widetilde g_{d,R,\eta})}
    \leq
    Cd^\alpha\e^{-\alpha}$. This completes the proof of Theorem~\ref{thm:elliptic-halfspace-fex}.
\end{proof}

\section{Conclusion}\label{sec:conclusion}
We establish that finite expressions provide a mathematically viable representation class for high-dimensional PDE solutions beyond the settings previously justified primarily through neural-network approximation. For several nonlinear and linear PDE classes, we prove that suitably constructed finite expressions achieve prescribed accuracy with evaluation cost growing only polynomially in the dimension and the reciprocal error. The central implication is that overcoming the curse of dimensionality does not fundamentally depend on large parametric architectures: structured symbolic representations generated from a fixed dictionary can possess comparable high-dimensional approximation power. These results place finite expression methods on a rigorous theoretical footing and suggest a broader program of developing symbolic representations for high-dimensional PDEs and scientific computing, including more general equations, operators, and constructive algorithms for discovering effective expressions.

\bibliography{refs_journal}

\end{document}